\documentclass[12pt,a4paper]{article}

\usepackage{amsmath,amsfonts,amsthm,amssymb}
\usepackage{mathrsfs} 
\usepackage{parskip}
\usepackage[T1]{fontenc}
\usepackage[utf8]{inputenc}
\usepackage{authblk}
\usepackage{enumerate}
\usepackage{graphicx}
\usepackage{xcolor}
\usepackage{a4wide}
\allowdisplaybreaks

\newtheorem{theorem}{Theorem}
\newtheorem{lemma}{Lemma}
\newtheorem{remark}{Remark}
\newtheorem{proposition}{Proposition}

\let\ds\displaystyle
\let\scr\mathscr
\let\goth\mathfrak

\def\zT{{\vphantom{\widetilde T} T}}
\def \Liminf{\mathop{\underline{\lim}}\limits}

\def\BB{\mathbb{B}}

\def\Pb{\mathbf{P}}

\def\Ex{\mathbf{E}}
\def\Pb{\mathbf{P}}

\def\KK{\mathbb{K}}

\def\UU{\mathbb{U}}

\def\1{\mbox{1\hspace{-.25em}I}}
\begin{document}
\title{Hidden AR Process and Adaptive Kalman Filter}
\author{ \textsc{Yury A. Kutoyants}\\ {\small Le Mans University, Le Mans,
    France }\\
{\small     National Research University ``MPEI'', Moscow, Russia }\\ 
{\small Tomsk     State University, Tomsk, Russia }\\
 }

\date{}

\maketitle
\begin{abstract}
 The model of partially observed linear system 
 depending on some unknown parameters is considered. An approximation of the
 unobserved component is proposed. This approximation is realized in three
 steps. First an estimator of the method of moments of unknown parameter is
 constructed. Then this estimator is used for defining the One-step
 MLE-process and finally the last estimator is substituted to the equations of
 Kalman  filter. The solution of obtained  equations provide us
 the approximation (adaptive K-B filter).  The asymptotic properties of all
 mentioned estimators and MLE and Bayesian estimators of the unknown
 parameters are described. The asymptotic efficiency of adaptive filtering is
 discussed.
\end{abstract}


\section{Introduction}

We are given  a  linear partially observed system
\begin{align}
\label{1-1}
X_t&=f\,Y_{t-1}+\sigma\, w_t,\qquad X_0,\qquad t=1,2,\ldots,\\
\label{1-2}  
Y_t&=a\,Y_{t-1}+b\, v_t,\qquad \;\;Y_0,
\end{align}
where $X^T=\left(X_0,X_1,\ldots,X_T\right)$ are observations and auto
regressive process (AR) $Y_t,t\geq 0$ is a hidden process. Here $w_t,t\geq 1$
and $v_t,t\geq 1$ are independent standard Gaussian random variables, i.e.,
$w_t\sim {\cal N}\left(0,1\right) $, $v_t\sim {\cal N}\left(0,1\right) $. The
initial values are $X_0\sim {\cal N}\left(0,d_x^2\right) $ and $Y_0\sim {\cal
  N}\left(0,d_y^2\right) $ and can be correlated with correlation $\rho
_{xy}=\Ex X_0Y_0$.  The system is defined by the parameters $a,b,f,\sigma^2,d_x^2,d_y^2
$.  It will be convenient for instant to denote $\vartheta
=\left(a,b,f,\sigma^2 \right)$.

 Denote ${\goth F}_t^X $ the $\sigma $-algebra generated by the first $t+1$
 observations $X_0,X_1,\ldots,X_t $. The conditional expectation $m\left(\vartheta
,t\right)=\Ex_\vartheta
 \left(Y_t| {\goth F}_t^X\right)$ according to the
equations of Kalman filter (see, e.g., Theorem 13.4 in \cite{LS01})  satisfies the equation
\begin{align}
\label{1-3}
m\left(\vartheta ,t\right)=a\,m\left(\vartheta
,t-1\right)+\frac{af\gamma\left(\vartheta,t-1 \right)}{\sigma
  ^2+f^2\gamma\left(\vartheta,t-1 \right)}\left[X_t-fm\left(\vartheta
  ,t-1\right)\right],\quad t\geq 1. 
\end{align}
  The initial value is $m\left(\vartheta ,0\right)=\Ex_\vartheta
 \left(Y_0|X_0\right) $.

 The  mean square  error $\gamma\left(\vartheta,t
 \right)=\Ex_\vartheta \left(Y_t-m\left(\vartheta ,t\right)\right)^2 $ is
 described by the equation 
\begin{align}
\label{1-4}
\gamma\left(\vartheta,t \right)=a^2\gamma\left(\vartheta,t-1
\right)+b^2-\frac{a^2f^2\gamma\left(\vartheta,t-1 \right)^2}{\sigma 
  ^2+f^2\gamma\left(\vartheta,t-1 \right) },\qquad \quad t\geq 1  
\end{align}
with the initial value $\gamma \left(\vartheta
,0\right)=\Ex_\vartheta \left(Y_0-m\left(\vartheta ,0\right)\right)^2  $.

We suppose that the observations $X^T$ are given and that some of the
parameters are unknown, but their values are always satisfy the condition
 \begin{align}
\label{A0}
 {\scr A}_0\;:\qquad a^2\in[0,1),\qquad b^2>0 ,\qquad\qquad f^2>0,\qquad \sigma
 ^2>0 .  
\end{align}
This condition is uniform in the following sense. If, for example, the unknown
parameter is $f\in \left(\alpha _f,\beta _f\right)$, then $\alpha _f>0$ or
$\beta _f<0$. 
  Our goal is to propose an approximation of $m\left(\vartheta
,t\right),t\geq 1$ in such situations and to describe the error of
approximation in the asymptotic of {\it large samples}, i.e., as $T\rightarrow
\infty $.

 The consistent estimation of the parameters
$\left(d_x^2,d_y^2\right) $ is impossible and we suppose that these parameters
are known. If these values are unknown, then the system
\eqref{1-3}-\eqref{1-4}  will be solved with some wrong initial values, but due to
robustness of the solutions the  difference between solutions with true and
wrong initial values under condition  ${\scr A}_0 $  is asymptotically negligible. 
Note as well that the consistent estimation of the parameters $\vartheta
=\left(f,b\right)$ or $\vartheta =\left(f,a,b\right) $ is impossible because
the model \eqref{1-1}-\eqref{1-2} depends on the product $fb$.

This work is devoted to the problem of estimation of $m\left(\vartheta
,t\right), t=1,\ldots,T$ in the situations, where some of the parameters of
the model \eqref{1-1}-\eqref{1-2} are unknown. As usual in such problems we
first estimate the unknown parameter and then this estimator is substituted in
the equations \eqref{1-3}-\eqref{1-4}. The obtained in such a way equations
will describe {\it adaptive Kalman filter}. There existe a wide literature on
adaptive filtering for such and similar partially observed systems. The
difference between them is in the construction of parameter estimators and in
the description of the corresponding errors of approximations of
$m\left(\vartheta ,\cdot \right)$, see, e.g.,
\cite{AWD10},\cite{BC93},\cite{BR85},\cite{CC87},\cite{Hay14},\cite{MS99},
\cite{S08},\cite{YG06},\cite{YL21}
and references there in. There is a large diversity of the models (linear and
non linear), different limits (small noise or large samples) and the methods
of adaptive filtering.  For the words ``adaptive Kalman filter'' Google
Scholar gives half million references. Of course not all of them are exactly
in what we need but nevertheless in some sens it gives the idea how important
this subject is.  We propose one else algorithm which realizes such
procedure. Note that since the work of Kalman \cite{Kal60} the equations
\eqref{1-1}-\eqref{1-2} are considered in more general forms, where $X_t$,
$Y_t$ and $w_t$, $v_t$ are vectors and $f,\sigma,a,b $ are matrices. Our
choice of this simplest model was motivated by the simplicity of calculations
and the same time the obtained results nevertheless seems to be non
trivial. We suppose that the proposed algorithms can be extended on more
complicate models and the results for these models will be similar to the
presented in this work ones.

Note that the problems of parameter estimation for such systems of observations
is a part of more general class of problems of parameter estimation for hidden
Markov processes, see, e.g., the works \cite{BRR98} and \cite{CMR05} and
references there in.

We are interested by the problem of on-line estimation of the conditional
expectation (random function) $m\left(\vartheta ,t\right), 0<t\leq T$ in the
different situations, where $\vartheta $ is unknown. For example,
$f=\vartheta $ and  $a$, $b$ are known.  The usual
behavior in such situations is to estimate first the unknown parameters and
then to substitute these estimators in the equations
\eqref{1-3}-\eqref{1-4}. The most interesting are of course the algorithms of
on-line recurrent adaptive filters. The studied algorithms are mainly verified
with the help of numerical simulations, which show the reasonable behavior of
the adaptive filters.

Our goal is to obtain a good recurrent approximation $m_t^\star,0<t\leq T$ of
the process $m\left(\vartheta ,t\right),0<t\leq T$ in the case of the
homogeneous partially observed system \eqref{1-1}-\eqref{1-2} and to discuss the
question of asymptotic efficiency of adaptive filters.

The estimation of   $m\left(\vartheta ,t\right),t\in (0,T]$ in this work
 is realized following the  program: 

{\it
\begin{enumerate}
\item  Calculate a preliminary estimator $\bar\vartheta_\tau  $ on relatively small
  interval of observations $\left[0,\tau \right]$.
\item Using $\bar\vartheta_\tau  $ construct the  One-step MLE-process $\vartheta
  _{t,T}^\star, \tau <t\leq T $.
\item As approximation of $m\left(\vartheta ,t\right) $ we propose $m_t^\star
  $ obtained with the help of K-B equations, where $\vartheta $ is replaced by
  $\vartheta _{t,T}^\star, \tau <t\leq T $.

\item Estimate the error $m _t^\star-m\left(\vartheta
 ,t\right), \tau <t\leq T  $.
\item Discuss the asymptotic efficiency of the adaptive filter.
\end{enumerate}
}

This means that we have no on-line approximation on the time interval
$\left[0,\tau \right]$, but $\tau /T\rightarrow 0$. Note that the used here
One-step MLE-process is the well-known Le Cam's One-step MLE, in which we 
consider the upper limit of the integral (time $t$) as variable. 

Introduce continuous time model 
\begin{align}
\label{1-6}
{\rm d}X_t&=f\left(\vartheta \right)\,Y_t\,{\rm d}t+\sigma \, {\rm
  d}W_t,\qquad\qquad \quad \; X_0,\quad \qquad 0\leq t\leq T,\\ 
{\rm  d}Y_t&=-a\left(\vartheta \right)\,Y_t\,{\rm d}t+b\left(\vartheta
\right)\,{\rm d}V_t,\qquad\qquad Y_0,\qquad\quad t\geq 0,\label{1-7}
\end{align}
where $W_t,V_t,t\geq 0$ are independent Wiener processes, $f\left(\vartheta
\right),a\left(\vartheta \right),b\left(\vartheta \right)$ are known smooth
functions and $\vartheta \in \Theta \subset{\cal R}^d$ is the unknown parameter.
Suppose that the observations are $X^T=\left(X_t,0\leq t\leq T\right)$ and the
Markov process $Y^T=\left(Y_t,0\leq t\leq T\right)$ is hidden.

We already applied this construction (steps 1-2, or steps 1-4, or steps 1-5) to 4
different models of observations. To  the model of continuous time observations like
\eqref{1-6}-\eqref{1-7} with small noises in the both equations \cite{Kut94},
\cite{KZ21} (1-4). To the model \eqref{1-6}-\eqref{1-7} with small noise in the
equation \eqref{1-6} only \cite{Kut19a}, \cite{Kut22} (1-5). To the model of
hidden telegraph process \cite{KhK18} (1-2). To  the model of observations
\eqref{1-6}-\eqref{1-7} in the asymptotics $T\rightarrow \infty $
\cite{Kut19b} (1-2),\cite{Kut23a} (1-5).

Note that in the Kalman filtering theory the model of observations is slightly
different  and can be written in our case as follows
\begin{align}
\label{1-8} 
X_t&=f\,Y_{t}+\sigma\, w_t,\qquad X_0,\qquad t=1,2,\ldots,\\
Y_{t+1}&=a\,Y_{t}+b\, v_{t+1},\qquad \;\;Y_0.
\label{1-9} 
\end{align}
The link between these two models and the modified equations
\eqref{1-3}-\eqref{1-4} can be found in \cite{LS01}, Corollary 3 of Theorem
13.4. Remark that if we consider the discrete time model as discrete time
approximation of the model \eqref{1-6}-\eqref{1-7}, then it seems the
equations \eqref{1-1}-\eqref{1-2} feet better than \eqref{1-8}-\eqref{1-9}.

 In this work we propose the adaptive filter for the model
 \eqref{1-1}-\eqref{1-2} (steps 1-5). The construction of preliminary method
 of moments estimators follows \cite{KhK18} and the exposition is in some
 sense similar to the exposition in the continuous time case of the work
 \cite{Kut23a}, where the model of observations is \eqref{1-6}-\eqref{1-7}.

In the next section we study the method of moments estimators (preliminary
estimators) of the parameters of the system \eqref{1}-\eqref{2}. Then the
different Fisher informations are calculated for different parameters (section
3).  Having preliminary estimator and Fisher information we introduce the
One-step MLE-processes and study their asymptotic properties (section 4). The
properties of MLE and Bayesian estimator are described in the section 5. The
One-step MLE-process is substituted in the equation \eqref{1-3}-\eqref{1-4}
and this provides us the adaptive filter (section 6). The last seventh section is
devoted to the question of asymptotic efficiency of the proposed adaptive filters.

\section{Method of moments estimators}

Introduce three statistics
\begin{align*}
S_{1,T}\left(X^T\right)&=\frac{1}{T}\sum_{t=1}^{T}\left(X_t-X_{t-1}\right)^2,\quad 
S_{2,T}\left(X^T\right)=\frac{1}{T}\sum_{t=2}^{T}\left(X_t-X_{t-1}\right)\left(X_{t-1}-X_{t-2}\right)
,\\ S_{3,T}\left(X^T\right)&=\frac{1}{T}\sum_{t=3}^{T}\left(X_t-X_{t-1}\right)\left(X_{t-2}-X_{t-3}\right)
\end{align*}
and study their asymptotic ($T\rightarrow \infty )$ behavior. We suppose
always that the condition ${\scr A}_0$ is fulfilled  and the true value is
denoted as 
$\vartheta _0$. 

Denote
\begin{align*}
\Phi_1 \left(\vartheta \right)=\frac{2f^2b^2}{1+a}+2\sigma ^2,\quad
\Phi_2 \left(\vartheta \right)= \frac{f^2b^2\left(a-1\right)}{1+a}-\sigma
^2,\quad \Phi_3 \left(\vartheta \right)=
\frac{f^2b^2a\left(a-1\right)}{\left(1+a\right)}. 
\end{align*}

\begin{lemma}
\label{L1}
We have the limits
\begin{align}
\label{12}
S_{1,T}\left(X^T\right)&\quad \longrightarrow \quad \Phi_1 \left(\vartheta_0
\right) ,\\
\label{13}
S_{2,T}\left(X^T\right)&\quad \longrightarrow \quad \Phi_2 \left(\vartheta_0
\right),\\
\label{14}
S_{3,T}\left(X^T\right)&\quad \longrightarrow \quad \Phi_3 \left(\vartheta_0
\right),
\end{align}
and  there exist constants $C_1>0$,
$C_2>0$, $C_3>0$ such that 
\begin{align}
\label{15}
&\Ex_{\vartheta _0}\left|S_{1,T}\left(X^T\right)-\Phi_1 \left(\vartheta_0
\right) \right|^2\leq\frac{ C_1}{T} , \quad \Ex_{\vartheta
  _0}\left|S_{2,T}\left(X^T\right)-\Phi_2 \left(\vartheta_0 \right)
\right|^2\leq \frac{C_2}{T},\\
\label{16}
& \Ex_{\vartheta
  _0}\left|S_{3,T}\left(X^T\right)-\Phi_3 \left(\vartheta_0 \right)
\right|^2\leq \frac{C_3}{T}.
\end{align}

\end{lemma}
\begin{proof} According to \eqref{1}
\begin{align*}
\frac{1}{T}\sum_{t=1}^{T}\left(X_t-X_{t-1}\right)^2&=\frac{f_0^2}{T}
\sum_{t=1}^{T}\left(Y_{t-1}-Y_{t-2}\right)^2+\frac{2f_0\sigma_0
}{T}\sum_{t=1}^{T}\left(Y_{t-1}-Y_{t-2}\right)\left(w_t-w_{t-1}\right)\\
&\qquad +\frac{\sigma_0
  ^2}{T}\sum_{t=1}^{T}\left(w_t-w_{t-1}\right)^2 .
\end{align*}
The Gaussian time series $Y_t,t\geq 1$ is exponentially mixing with the
stationary (invariant) Gaussian distribution ${\cal N}\left(0,
\frac{b_0^2}{1-a_0^2}\right)$ and it is independent of $w_t,t\geq 1
$. Therefore by the law of large numbers we have the convergences 
\begin{align*}
&\frac{f_0^2}{T}
\sum_{t=1}^{T}\left(Y_{t-1}-Y_{t-2}\right)^2=\frac{f_0^2}{T}
\sum_{t=1}^{T}\left(\left(a_0-1\right)Y_{t-2}+b_0v_{t-1}\right)^2\\
&\qquad \qquad \qquad=\frac{f_0^2\left(1-a_0\right)^2}{T}
\sum_{t=1}^{T}Y_{t-2}^2-\frac{2b_0\left(1-a_0\right)f_0^2}{T}
\sum_{t=1}^{T}Y_{t-2}  v_{t-1}+\frac{f_0^2b_0^2}{T}
\sum_{t=1}^{T}v_{t-1}^2\\
&\qquad \qquad \qquad \longrightarrow
\frac{f_0^2b_0^2\left(1-a_0\right)^2}{1-a_0^2}+f_0^2b_0^2=\Phi
_1\left(\vartheta _0\right)-2\sigma _0^2,\\ 
&\frac{2f_0\sigma_0
}{T}\sum_{t=1}^{T}\left(Y_{t-1}-Y_{t-2}\right)\left(w_t-w_{t-1}\right)\longrightarrow
0,\\
&\frac{\sigma_0
  ^2}{T}\sum_{t=1}^{T}\left(w_t-w_{t-1}\right)^2\longrightarrow 2\sigma_0
  ^2,
\end{align*}
which proves \eqref{12}.

To prove \eqref{13} we write
\begin{align*}
&\left(X_t-X_{t-1}\right)\left(X_{t-1}-X_{t-2}\right)\\
&\qquad =\left[f_0\left(Y_{t-1}-Y_{t-2}\right)+
  \sigma_0 \left(w_t-w_{t-1}\right)\right]\left[f_0\left(Y_{t-2}-Y_{t-3}\right)+
  \sigma_0 \left(w_{t-1}-w_{t-2}\right)\right]\\
&\qquad =\left[f_0\left(a_0-1\right)Y_{t-2}+f_0b_0v_{t-1}+
  \sigma_0 \left(w_t-w_{t-1}\right)\right]\\
&\qquad \qquad \qquad \times
  \frac{1}{a_0}\left[f_0\left(a_0-1\right)Y_{t-2}+f_0b_0a_0v_{t-2}+ 
  \sigma_0 a_0\left(w_{t-1}-w_{t-2}\right)\right].
\end{align*}
Therefore
\begin{align*}
\frac{1}{T}\sum_{t=2}^{T}\left(X_t-X_{t-1}\right)\left(X_{t-1}-X_{t-2}\right)
 & =\frac{f_0^2\left(a_0-1\right)^2 }{a_0T}\sum_{t=2}^{T}Y_{t-2}^2+\frac{f_0^2
  b_0\left(a_0-1\right)}{T}\sum_{t=2}^{T}   Y_{t-2} v_{t-2}                    \\
&\qquad +\frac{\sigma_0
   ^2}{T}\sum_{t=2}^{T}\left(w_t-w_{t-1}\right)
\left(w_{t-1}-w_{t-2}\right)+o\left(1\right). 
\end{align*}
We have
\begin{align*}
&\frac{1}{T}\sum_{t=2}^{T} Y_{t-2} v_{t-2}  =\frac{a_0}{T}\sum_{t=2}^{T} Y_{t-3}
v_{t-2} +\frac{a_0b_0}{T}\sum_{t=2}^{T}  v_{t-2}^2 \longrightarrow a_0b_0,\\
&\frac{1}{T}\sum_{t=2}^{T}\left(w_t-w_{t-1}\right)
\left(w_{t-1}-w_{t-2}\right)\longrightarrow -1.
\end{align*}
Finally
\begin{align*}
\frac{1}{T}\sum_{t=2}^{T}\left(X_t-X_{t-1}\right)\left(X_{t-1}-X_{t-2}\right)&\longrightarrow
\frac{f_0^2b_0^2\left(a_0-1\right)^2 }{a_0\left(1-a_0^2\right) } +f_0^2
  b_0^2a_0\left(a_0-1\right)-\sigma ^2\\
&=\Phi_2 \left(\vartheta_0 \right).
\end{align*}

The last convergence \eqref{14} we obtain using the similar arguments as follows:
\begin{align*}
&\left(X_t-X_{t-1}\right)\left(X_{t-2}-X_{t-3}\right) \\ 
&\qquad =\left[f_0\left(Y_{t-1}-Y_{t-2}\right)+\sigma_0 \left(w_t-
    w_{t-1}\right)\right] \left[f_0\left(Y_{t-3}-Y_{t-4}\right)+\sigma_0 \left(w_{t-2}-
    w_{t-3}\right)\right]\\ 
&\qquad =\left[f_0\left(a_0Y_{t-2}-Y_{t-2}\right)+f_0b_0v_{t-1}+\sigma_0 \left(w_t-
    w_{t-1}\right)\right]\\
&\qquad \qquad \times
  \frac{1}{a_0}\left[f_0\left(a_0Y_{t-3}-Y_{t-3}\right)+{f_0b_0}v_{t-3}+a_0\sigma_0
    \left(w_{t-2}-     w_{t-3}\right)\right].
\end{align*}
Here we used the equation \eqref{2} and the equality
$a_0Y_{t-4}=Y_{t-3}-b_0v_{t-3}$. Further, as $w_t,t\geq 1$ are independent of
$Y_t,t\geq 1$ and $v_t,t\geq 1$ we can write
\begin{align*}
S_{3,T}\left(X^T\right)&=\frac{1}{a_0T}\sum_{t=3}^{T}\left[f_0\left(a_0-1\right)Y_{t-2}
  +f_0b_0v_{t-1}\right]\left[f_0\left(a_0-1\right)Y_{t-3}+f_0b_0v_{t-3}\right]+o\left(1\right)\\ 
&=\frac{1}{a_0T}\sum_{t=3}^{T}\left[f_0\left(a_0-1\right)a_0Y_{t-3}+f\left(a_0-1\right)v_{t-2}
  +f_0b_0v_{t-1}\right]\\ &\qquad \quad \qquad \quad \qquad \quad
\times\left[f_0\left(a_0-1\right)Y_{t-3}+f_0b_0v_{t-3}\right]+o\left(1\right)\\
 &=\frac{f_0^2\left(a_0-1\right)^2}{T}\sum_{t=3}^{T}Y_{t-3}^2+\frac{f_0^2\left(a_0-1\right)b_0}{T}\sum_{t=3}^{T}Y_{t-3}v_{t-3}+o\left(1\right)\\  
&=\frac{f_0^2\left(a_0-1\right)^2}{T}\sum_{t=3}^{T}Y_{t-3}^2+\frac{f_0^2b_0^2\left(a_0-1\right)}{T}\sum_{t=3}^{T}v_{t-3}^2+o\left(1\right)\\ 
&\longrightarrow
\frac{f_0^2b_0^2\left(1-a_0\right)}{1+a_0}+f_0^2b_0^2\left(a_0-1\right)=\Phi
_3\left(\vartheta _0\right).
\end{align*}

As the Gaussian AR process $Y_t,t\geq 1$ has exponentially decreasing
correlation function the convergences \eqref{15},\eqref{16} follow from
the standard arguments. For the higher moments see Rosenthal-type inequalities
\cite{DD07}.

\end{proof}

\begin{remark}
\label{R1}
{\rm From the proofs it follows that the estimates \eqref{14} are valid
  uniformly on compacts $\KK\subset\Theta $ too, i.e., 
\begin{align}
\label{17}
&\sup_{\vartheta_0 \in\KK}\Ex_{\vartheta _0}\left|S_{1,T}\left(X^T\right)-\Phi_1 \left(\vartheta_0
\right) \right|^2\leq\frac{ C}{T} , \quad \sup_{\vartheta_0 \in\KK}\Ex_{\vartheta
  _0}\left|S_{2,T}\left(X^T\right)-\Phi_2 \left(\vartheta_0 \right)
\right|^2\leq \frac{C}{T}\\
\label{18}
&\sup_{\vartheta_0 \in\KK}\Ex_{\vartheta _0}\left|S_{3,T}\left(X^T\right)-\Phi_3\left(\vartheta_0
\right) \right|^2\leq\frac{ C}{T} .
\end{align}

}
\end{remark}
\begin{remark}
\label{R2}
{\rm More detailed analysis allows to verify the asymptotic normality
\begin{align*}
&\sqrt{T}\left(S_{1,T}\left(X^T\right)-\Phi_1 \left(\vartheta_0
\right) \right)\Longrightarrow {\cal N}\left(0,D_1\left(\vartheta
_0\right)^2\right),\\
&\sqrt{T}\left(S_{2,T}\left(X^T\right)-\Phi_2 \left(\vartheta_0
\right) \right)\Longrightarrow {\cal N}\left(0,D_2\left(\vartheta _0\right)^2\right),\\
&\sqrt{T}\left(S_{3,T}\left(X^T\right)-\Phi_3 \left(\vartheta_0
\right) \right)\Longrightarrow {\cal N}\left(0,D_3\left(\vartheta _0\right)^2\right)
\end{align*}
but we do not prove these convergences because we need these MMEs just for
construction of One-step MLE-processes and the estimates \eqref{17},\eqref{18} are
sufficient for these problems.

}
\end{remark}

All parameters of the model \eqref{1}-\eqref{2} can be estimated with
the help of the introduced   statistics $
S_{1,T}\left(X^T\right),S_{2,T}\left(X^T\right),S_{3,T}\left(X^T\right)$. Below
 the method of moments estimators (MME) of the   parameters
$f,a,b,\sigma ^2$ are proposed and their asymptotic behavior is described.

\subsection{Estimation of the parameter $f$.}

 Suppose that the parameters $a,b,\sigma ^2$ are known
and we have to estimate $\vartheta =f\in\left(\alpha _f,\beta _f\right)$, $\alpha
_f>0$. Then the MME can be defined as follows
\begin{align}
\label{19}
f_T^*= \alpha _f \1_{\left\{\BB_{1,T}\right\}}+\bar
f_T\1_{\left\{\BB_{2,T}\right\}}+\beta _f \1_{\left\{\BB_{3,T}\right\}}. 
\end{align}
Here
\begin{align*}
\bar f_T&= \left(\frac{\left(S_{1,T}\left(X^T\right)-2\sigma
  ^2\right)\left(1+a\right)}{2b^2}\right)^{1/2},\\ 
\BB_{1,T}&=\left\{{\rm The
  \; event:}\quad S_{1,T}\left(X^T\right)\leq \frac{2\alpha
  _f^2b^2}{1+a}+2\sigma ^2 \right\},\\  
\BB_{2,T}&=\left\{{\rm The
  \; event:}\quad \frac{2\alpha _f^2b^2}{1+a}+2\sigma
^2<S_{1,T}\left(X^T\right)\leq \frac{2\beta _f^2b^2}{1+a}+2\sigma ^2
\right\},\\  
\BB_{3,T}&=\left\{{\rm The
  \; event:}\quad S_{1,T}\left(X^T\right)\geq \frac{2\beta
  _f^2b^2}{1+a}+2\sigma ^2 \right\} .
\end{align*}
Therefore  $f_T^*\in \left[\alpha _f,\beta _f\right]$. 
As 
\begin{align*}
S_{1,T}\left(X^T\right)\longrightarrow \Phi _1\left(\vartheta
_0\right)=\frac{2f_0^2b^2}{1+a}+2\sigma ^2 
\end{align*}
the probabilities
\begin{align*}
\Pb_{f_0}\left(\BB_{1,T}\right)\longrightarrow 0,\qquad \Pb_{f_0}\left(\BB_{2,T}\right)\longrightarrow 1,\quad    \Pb_{f_0}\left(\BB_{3,T}\right)\longrightarrow 0,
\end{align*}
and below we  omit the representations like  \eqref{19} for the other MMEs. 

It is easy to see that by Lemma \ref{L1} and Remark \ref{R1} the MME
$f_T^* $ is consistent, i.e., $f_T^*\rightarrow f_0$. 

Let us verify the upper bound
\begin{align*}
\sup_{f_0\in\KK}\Ex_{f_0}\left|f_T^*-f_0 \right|^2\leq \frac{C}{T}.
\end{align*}
Put $\eta _T=\sqrt{T}\left(S_{1,T}\left(X^T\right)-\Phi _1\left(\vartheta
_0\right)\right)$.
Note that thanks to the definition \eqref{19} of $f_T^*$ it is sufficient
to study the  statistic $S_{1,T}\left(X_T\right)$ on the set $\BB_2$ only and
therefore we have the estimates
\begin{align}
\label{21}
\frac{2\alpha _f^2b^2}{1+a}+2\sigma ^2 \leq S_{1,T}\left(X_T\right)\leq
\frac{2\beta _f^2b^2}{1+a}+2\sigma ^2  .
\end{align}

 Then
\begin{align*}
f_T^*&=\sqrt{\frac{1+a}{2b^2}}\left[\Phi _1\left(\vartheta
  _0\right)-2\sigma ^2+T^{-1/2}\eta _T \right]^{1/2}\\
&=\sqrt{\frac{1+a}{2b^2}}\left[\Phi _1\left(\vartheta
  _0\right)-2\sigma ^2\right]^{1/2}+ \sqrt{\frac{1+a}{2b^2}}\left[\Phi _1\left(\vartheta
  _0\right)-2\sigma ^2+sT^{-1/2}\eta _T \right]^{-1/2}T^{-1/2}\eta _T\\
&=f_0+\sqrt{\frac{1+a}{2b^2}}\left[\Phi _1\left(\vartheta
  _0\right)-2\sigma ^2+sT^{-1/2}\eta _T \right]^{-1/2}T^{-1/2}\eta _T,
\end{align*}
where  $s\in \left(0,1\right)$   and therefore 
\begin{align*}
\left|f_T^*-f_0\right|\leq \frac{\left(1+a\right)}{2b^2\alpha _f}\;T^{-1/2}\eta _T. 
\end{align*}
Here we used \eqref{21}. Therefore, by Lemma \ref{L1}
\begin{align*}
 \sup_{f_0\in\KK} \Ex_{f_0}\left|f_T^*-f_0 \right|^2\leq C
 \,T^{-1}\Ex_{f_0}\left|\eta _T \right|^2.
\end{align*}

\begin{remark}
\label{R3}
{\rm Of course, the statistics $S_{2,T}\left(X^T\right) $ and
  $S_{3,T}\left(X^T\right) $ as well can be used for the construction of the
  MME of $f$. For example, we can solve the equation
\begin{align*}
S_{2,T}\left(X^T\right)= \frac{f^2b^2\left(a-1\right)}{1+a}-\sigma ^2
\end{align*}
with respect to $f$ and to put
\begin{align*}
f_T^{**}=\left(\frac{\left(S_{2,T}\left(X^T\right)+ \sigma
  ^2\right)\left(1+a\right)}{b^2\left(a-1\right)}\right)^{1/2}. 
\end{align*}
Note that the statistic $S_{2,T}\left(X^T\right) $ takes negative values and
the expression under square root brackets is positive with probability tending
to 1. This MME has the  asymptotic  properties similar to that of  $f_T^{*}$.

}
\end{remark}
\subsection{Estimation of the parameter $b$.}

The estimation of $b$ is almost the same as the estimation of $f$ because
these parameters in $\Phi _1\left(\vartheta \right)$, $\Phi
_2\left(\vartheta \right)$ and $\Phi
_3\left(\vartheta \right)$ are in the products $bf$ only. Therefore the
MME of $b$ has the properties
\begin{align}
\label{22}
b_T^*&= \left(\frac{\left(S_{1,T}\left(X^T\right)-2\sigma
  ^2\right)\left(1+a\right)}{2f^2}\right)^{1/2}\longrightarrow b_0,\qquad
\sup_{b_0\in\KK}\Ex_{b_0}\left|b_T^*-b_0 \right|^2\leq \frac{C}{T}
\end{align}
{\it et ctr.}

\subsection{Estimation of the parameter $a$.}

The solution of the equation
\begin{align*}
S_{1,T}\left(X^T\right)=\frac{2f^2b^2}{1+a}+2\sigma ^2
\end{align*}
leads to the MME
\begin{align*}
a_T^*=\frac{2f^2b^2}{S_{1,T}\left(X^T\right)-2\sigma ^2}-1.
\end{align*}
By Lemma \ref{L1} we have
\begin{align*}
a_T^*\longrightarrow a_0,\qquad \sup_{a_0\in\KK}\Ex_{a_0}\left|a_T^*-
a_0\right|^2\leq \frac{C}{T}. 
\end{align*}

\subsection{Estimation of the parameter $\sigma ^2$.}

The MME
\begin{align*}
\sigma ^{2*}_T=\frac{1}{2}S_{1,T}\left(X^T\right)-\frac{f^2b^2}{1+a}
\end{align*}
is consistent and 
\begin{align*}
\sup_{\sigma ^2_0\in\KK}\Ex_{\sigma ^2_0}\left|\sigma ^{2*}_T-
\sigma ^2_0\right|^2\leq \frac{C}{T}. 
\end{align*}

\subsection{Estimation of the parameter $\vartheta =\left(a,f\right)$.}

The MME $\vartheta _T^*=\left(a_T^*,f_T^*\right)$ is solution of the system of
equations
\begin{align*}
S_{1,T}\left(X^T\right)=\Phi _1\left(\vartheta _T^*\right),\qquad \quad
S_{2,T}\left(X^T\right)=\Phi _2\left(\vartheta _T^*\right) 
\end{align*}
and has the  following form
\begin{align}
\label{24}
a_T^*&=\frac{S_{1,T}\left(X^T\right)+S_{2,T}\left(X^T\right)-\sigma
  ^2}{S_{1,T}\left(X^T\right)-2\sigma ^2} ,\\
\label{25}
f_T^*&=\left(\frac{S_{1,T}\left(X^T\right)\left(1+a_T^*\right)-2\sigma ^2
}{2b^2}\right)^{1/2}. 
\end{align}
With the help of Lemma \ref{L1} it can be shown that 
\begin{align*}
\left(a_T^*,f_T^*\right)\longrightarrow \left(a_0,f_0\right)
\end{align*}
and
\begin{align}
\label{26}
\Ex_{\vartheta _0}\left\|\vartheta ^{*}_T-\vartheta_0\right\|^2\leq \frac{C}{T}. 
\end{align}

\subsection{Estimation of the parameter $\vartheta =\left(a,f,\sigma ^2\right)$.}

In this case we have three equations
\begin{align*}
S_{1,T}\left(X^T\right)&=\frac{2f^2b^2}{1+a}+2\sigma ^2,\quad
S_{2,T}\left(X^T\right)= \frac{f^2b^2\left(a-1\right)}{1+a}-\sigma
^2,\\ S_{3,T}\left(X^T\right)&=
\frac{f^2b^2a\left(a-1\right)}{\left(1+a\right)}.
\end{align*}
The MME $\vartheta _T^*=\left(a_T^*,f_T^*,\sigma_T^{2*}\right)$ is the
following   solution of this system:
\begin{align*}
a_T^*&=\frac{2S_{3,T}\left(X^T\right)}{S_{1,T}\left(X^T\right)+2S_{2,T}\left(X^T\right)}+1,\\
f_T^*&=\frac{S_{3,T}\left(X^T\right)\left(1+a_T^*\right)}{b^2a_T^*\left(a_T^*-1\right)},\\
\sigma_T^{2*}&=\frac{1}{2}S_{1,T}\left(X^T\right)-\frac{\left(f_T^*\right)^2b^2}{1+a_T^*}.
\end{align*}
Once more we have the consistency of $\vartheta _T^*$ and the bound like
\eqref{26}.

Almost similar result we have in the case of estimation $\vartheta
=\left(a,b,\sigma ^2\right)$. 

\bigskip

\begin{remark}
\label{R4}
{\rm Of course, it is possible to define and study the MMEs of the parameters
  $\vartheta =\left(a,b\right)$, $\vartheta =\left(f,\sigma ^2\right)$ and
  $\vartheta =\left(b,\sigma ^2\right)$.  The only forbidden couple of
  parameters is $\vartheta =\left(f,b\right)$. In this case the consistent
  estimation of $\vartheta $ is impossible. Indeed, the stationary AR process
  $Y_t,t=\ldots, -1,0,1,\ldots $ admits the representation
\begin{align*}
Y_t=b\sum_{k=0}^{\infty }a^kv_{t-k}
\end{align*}
and therefore the observed process is 
\begin{align*}
X_t=fb\sum_{k=0}^{\infty }a^kv_{t-k}+\sigma w_t,\qquad k=...,-1,0,1,\ldots.
\end{align*}
Here we introduced the sequence $ v_t, t=...,-1,0,1,\ldots$ of i.i.d. standard
Gaussian r.v.'s.
 The model depends on the product $fb$ and therefore these parameters
can not be estimated separately. 
}
\end{remark}

\section{Fisher informations}

We have the same model of observations
\begin{align*}
X_t&=fY_{t-1}+\sigma \,w_t,\qquad X_0,\qquad t\geq 1,\\
Y_t&=aY_{t-1}+b\,v_{t},\qquad Y_0,\qquad t\geq 1,
\end{align*}
where $w_t,v_t,t\geq 1$ are independent standard Gaussian r.v.'s and $f,\sigma^2
,a,b$ are parameters of the model. As before, we suppose that some of these
parameters are unknown and we have to estimate the unknown parameters  $\vartheta \in\Theta $ by the
observations $X^T=\left(X_0,X_1,\ldots,X_T\right)$. 

The MMEs studied above are consistent, but not asymptotically efficient. That
is why we propose below the construction of One-step MLE-process, which allow
us to solve two problems: first we obtain asymptotically efficient estimators
of these parameters and the second - we describe the approximation of the
conditional expectation $m\left(\vartheta ,t\right),t\geq 1$. 

The construction of One-step MLE-processes requires the knowledge of Fisher
information. That is why we calculate below the Fisher informations related
with different parameters. To do this we recall first  some known properties of
$\gamma \left(\vartheta ,t\right)$ and describe the likelihood ratio function
for this model.

Consider the model of partially observed time series 
\begin{align}
\label{1}
X_t&=f\,Y_{t-1}+\sigma\, w_t,\qquad X_0,\qquad t=1,2,\ldots,\\
\label{2}  
Y_t&=a\,Y_{t-1}+b\, v_t,\qquad \;\;Y_0,
\end{align}
where $X^T=\left(X_0,X_1,\ldots,X_T\right)$ are observations and auto
regressive process (AR) $Y_t,t\geq 0$ is a hidden process. Here $w_t,t\geq 1$
and $v_t,t\geq 1$ are independent standard Gaussian random variables, i.e.,
$w_t\sim {\cal N}\left(0,1\right) $, $v_t\sim {\cal N}\left(0,1\right) $. The
initial values are $X_0\sim {\cal N}\left(0,d_x^2\right) $ and $Y_0\sim {\cal
  N}\left(0,d_y^2\right) $.  The system is defined by the parameters
$a,b,f,\sigma^2,d_x^2,d_y^2 $. We suppose that some of these parameters are
unknown and have to be estimated by observations $X^T$.

It will be convenient for instant to denote $\vartheta =\left(a,b,f,\sigma^2
\right)$. 

 Denote ${\goth F}_t^X $ the $\sigma $-algebra generated by the first $t+1$
 observations $X_0,X_1,\ldots,X_t $. The conditional expectation
 $m\left(\vartheta ,t\right)=\Ex_\vartheta \left(Y_t| {\goth F}_t^X\right)$
 according to the equations of Kalman filter (see, e.g., \cite{Kal60},
 \cite{LS01}) satisfies the equation
\begin{align}
\label{4}
m\left(\vartheta ,t\right)=a\,m\left(\vartheta
,t-1\right)+\frac{af\gamma\left(\vartheta,t-1 \right)}{\sigma
  ^2+f^2\gamma\left(\vartheta,t-1 \right)}\left[X_t-fm\left(\vartheta
  ,t-1\right)\right],\quad t\geq 1. 
\end{align}
  The initial value is $m\left(\vartheta ,0\right)=\Ex_\vartheta
 \left(Y_0|X_0\right) $.

 The  mean square  error $\gamma\left(\vartheta,t
 \right)=\Ex_\vartheta \left(Y_t-m\left(\vartheta ,t\right)\right)^2 $ is
 described by the equation 
\begin{align}
\label{5}
\gamma\left(\vartheta,t \right)=a^2\gamma\left(\vartheta,t-1
\right)+b^2-\frac{a^2f^2\gamma\left(\vartheta,t-1 \right)^2}{\sigma 
  ^2+f^2\gamma\left(\vartheta,t-1 \right) },\qquad \quad t\geq 1  
\end{align}
with the initial value $\gamma \left(\vartheta
,0\right)=\Ex_\vartheta \left(Y_0-m\left(\vartheta ,0\right)\right)^2  $.

If some of the mentioned parameters are unknown, then, of course, we can not
use \eqref{4}-\eqref{5} for calculation of $m\left(\vartheta
,t\right),t\geq 1$. 

We suppose that the
observations $X^T=\left(X_0,X_1,\ldots,X_T\right)$ are given and that some of
the parameters are unknown, but their values are always satisfy the condition
${\scr A}_0$.  Our goal is to propose an approximation of
$m\left(\vartheta ,t\right),t\geq 1$ in such situations and to describe the
error of approximation in the asymptotic of {\it large samples}, i.e., as
$T\rightarrow \infty $.

 The consistent estimation of the parameters
$\left(d_x^2,d_y^2\right) $ is impossible and we suppose that these parameters
are known. If these values are unknown, then the system
\eqref{4}-\eqref{5}  will be solved with some wrong initial values, but due to
robustness of the solutions the  difference between solutions with true and
wrong initial values in our problems  is asymptotically negligible.

  As before, the proposed program is consists in several
steps. First on some learning interval $\left[0,\tau _\zT\right]$ of
negligible length ($\tau _\zT/T\rightarrow 0$) we construct a consistent
preliminary estimator $\vartheta _{\tau _\zT}^*$. Then this estimator is used
for defining the One-step MLE-process $\vartheta ^\star_T= \left(\vartheta
_{t,T}^\star, t=\tau _\zT+ 1,\ldots,T\right)$ and finally the approximation
$m_{T}^\star=\left(   m_{t,T}^\star, t= \tau _\zT+ 1,\ldots,T\right)$ is obtained by substituting $\vartheta
^\star_T $ in the equations \eqref{4}-\eqref{5}. The last step is to
evaluate the error $m_{t,T}^\star-m\left(\vartheta ,t\right) $.

Remark that the function $\gamma\left(\vartheta,t \right) $ converges to the
value 
\begin{align}
\label{6}
\gamma _*\left(\vartheta \right)=\frac{f^2b^2-\sigma ^2\left(1-a^2\right)}{2f^2}
+\frac{1}{2}\left[\left(\frac{\sigma
  ^2\left(1-a^2\right)}{f^2}-b^2\right)^2+\frac{4b^2\sigma ^2}{f^2}\right]^{1/2}
\end{align}
as $t\rightarrow \infty $ (see Example 3 in section 14.4, \cite{LS01}).  The
value $\gamma _*\left(\vartheta \right) $ is obtained as a positive solution
of the equation \eqref{5}, where we put $\gamma\left(\vartheta,t
\right)=\gamma\left(\vartheta,t-1 \right)= \gamma _*\left(\vartheta \right)$,
which becomes
\begin{align*}
\gamma _*\left(\vartheta \right)^2+\left[\frac{\sigma
    ^2\left(1-a^2\right)}{f^2}-b^2\right] \gamma _*\left(\vartheta
\right)-\frac{b^2\sigma ^2}{f^2}=0. 
\end{align*}

Below we study the asymptotic ($T\rightarrow \infty $) properties of
estimators. That is why to simplify the exposition we suppose that the initial
value $\gamma\left(\vartheta,0 \right)= \gamma _*\left(\vartheta
\right)$. Then for any $t\geq 1$ we have $\gamma\left(\vartheta,t \right)=
\gamma _*\left(\vartheta \right) $. Of course, this is condition on
correlation between $X_0$ and $Y_0$ and the values $d_x^2,d_y^2$.

Therefore the  equation \eqref{4} is replaced by the equation 
\begin{align}
\label{7}
m_t\left(\vartheta \right)=a\,m_{t-1}\left(\vartheta
\right)+\frac{af\gamma_*\left(\vartheta \right)}{\sigma
  ^2+f^2\gamma_*\left(\vartheta \right)}\left[X_t-fm_{t-1}\left(\vartheta
\right)\right],\quad t\geq 1
\end{align}
with the corresponding initial value, providing $\gamma\left(\vartheta,0 \right)=
\gamma _*\left(\vartheta \right) $.  Recall that
 equation \eqref{4}    is stable w.r.t. the initial value, i.e., for the wrong
initial condition the difference  $m\left(\vartheta,t
\right)-m_t\left(\vartheta \right)\rightarrow 0 $. 

Note as well that if we denote $\vartheta _0$ the true value, then
\begin{align*}
\zeta _t\left(\vartheta _0\right)=\frac{X_t-f_0m_{t-1}\left(\vartheta
  _0\right)}{\sqrt{\sigma _0^2+f_0^2\gamma _*\left(\vartheta _0\right)}}
,\qquad t\geq 1
\end{align*}
are  i.i.d. standard Gaussian random variables (see Theorem 13.5 in
\cite{LS01}). This means that 
the equation of observations \eqref{1} can be written as follows
\begin{align*}
X_t=f_0m_{t-1}\left(\vartheta _0\right)+ \sqrt{\sigma _0^2+f_0^2\gamma
  _*\left(\vartheta _0\right)}\,\zeta _t\left(\vartheta _0\right),\qquad t\geq
1.
\end{align*}
Using this representation we can rewrite  the equation \eqref{7} too
\begin{align}
\label{9}
m_t\left(\vartheta \right)&=a\,m_{t-1}\left(\vartheta
\right)+\frac{af\gamma_*\left(\vartheta \right)}{\sigma
  ^2+f^2\gamma_*\left(\vartheta \right)}\left[f_0m_{t-1}\left(\vartheta
  _0\right)-fm_{t-1}\left(\vartheta
  \right)\right]\nonumber\\
&\qquad +\frac{af\gamma_*\left(\vartheta \right)\sqrt{\sigma
    _0^2+f_0^2\gamma _*\left(\vartheta _0\right)}}{\sigma
  ^2+f^2\gamma_*\left(\vartheta \right)}\zeta _{t}\left(\vartheta _0\right),\quad t\geq 1.
\end{align}

The likelihood  function is
\begin{align}
\label{10}
L\left(\vartheta ,X^T\right)&=\left(2\pi\left(\sigma^2+f^2\gamma
_*\left(\vartheta \right) \right)
\right)^{-T/2}\exp\left(-\frac{1}{2}\sum_{t=1}^{T}\frac{\left(X_t-fm_{t-1}\left(\vartheta
  \right)\right)^2}{{\sigma^2+f^2\gamma _*\left(\vartheta \right)}}
\right)\nonumber\\ 
&=\left(\frac{1}{2\pi P \left(\vartheta
  \right)}\right)^{T/2}\exp\left(-\frac{1}{2}\sum_{t=1}^{T}\frac{\left(X_t-fm_{t-1}\left(\vartheta
  \right)\right)^2}{{ P \left(\vartheta \right) }} \right),\qquad \vartheta
\in\Theta .
\end{align}

Here 
\begin{align*}
P \left(\vartheta \right)= \sigma^2+f^2\gamma _*\left(\vartheta \right) 
\end{align*}
and $\Theta $ is an open, bounded, convex set of the possible values of the
parameter $\vartheta $.

\subsection{Unknown parameter $b$}
We start with one-dimensional case, say, $\vartheta =b\in\Theta =\left(\alpha
_b,\beta _b\right), \alpha _b>0$.  As above we suppose that $f\not=0,a^2\in
[0,1)$. Therefore the system is 
\begin{align*}
X_t&=f\,Y_{t-1}+\sigma\, w_t,\qquad X_0,\qquad t=1,2,\ldots,\\
Y_t&=a\,Y_{t-1}+\vartheta \, v_t,\qquad \;\;Y_0,
\end{align*}
Introduce the notation: 
\begin{align*}
&\Gamma \left(\vartheta \right)=f^2\gamma _*\left(\vartheta
  \right),\qquad  P \left(\vartheta\right)=\sigma ^2+\Gamma\left(\vartheta
  \right),\qquad A\left(\vartheta \right) =\frac{a\sigma ^2}{ P \left(\vartheta\right)},\\
 &
  B\left(\vartheta,\vartheta_0\right)=\frac{a\Gamma\left(\vartheta \right)\sqrt{ P \left(\vartheta_0 \right)}}{ P
    \left(\vartheta \right)},\qquad\dot  B_b\left(\vartheta_0,\vartheta_0\right)=\left.\frac{\partial
    B\left(\vartheta,\vartheta_0\right)}{\partial \vartheta
  }\right|_{\vartheta =\vartheta _0}= \frac{a\sigma ^2\dot\Gamma\left(\vartheta_0 \right)}{ P \left(\vartheta_0\right)^{3/2} }.
\end{align*}
Note that $\inf_{\vartheta \in\Theta }\dot\Gamma
      \left(\vartheta_0 \right)>0 $ (see \eqref{52} below). Another
estimate which we will use is 
\begin{align}
\label{27}
\sup_{\vartheta \in\Theta }\left|A\left(\vartheta \right) \right|<1.
\end{align}
Let us show the well known fact that it is always fulfilled if $a^2\in [0,1 ) $. We have
\begin{align*}
 P \left(\vartheta \right)&=\frac{f^2\vartheta^2+\sigma ^2+\sigma ^2a^2}{2}  
+\frac{1}{2}\left[\left(\sigma
  ^2\left(1-a^2\right)-f^2\vartheta^2\right)^2+ 4f^2\vartheta^2\sigma ^2 \right]^{1/2}
\end{align*}
and
\begin{align*}
\left|A\left(\vartheta \right)\right|&={2\left|a\right|\sigma ^2}\left({f^2\vartheta^2+\sigma ^2+\sigma ^2a^2  
+ \left[\left(\sigma
  ^2\left(1-a^2\right)-f^2\vartheta^2\right)^2+ 4f^2\vartheta^2\sigma ^2\right]^{1/2}}\right)^{-1}\\
&< \frac{2\left|a\right|\sigma ^2}{\sigma ^2+\sigma ^2a^2
}=\frac{2\left|a\right|}{1+a^2}<1. 
\end{align*}

\begin{proposition}
\label{P1}  The Fisher information is 
\begin{align}
\label{28}
{\rm I}_b\left(\vartheta _0\right)&=\frac{\dot P_b\left(\vartheta_0
\right)^2 \left[P\left(\vartheta_0
\right)^2   +a^2\sigma ^4 \right]  }{2 P
  \left(\vartheta_0\right)^2 \left[P\left(\vartheta_0
\right)^2  -a^2\sigma ^4 \right]}.
\end{align}
\end{proposition}
\begin{proof}

 The equation \eqref{9} we multiply by $f$, denote $M_t\left(\vartheta
 \right)=fm_t\left(\vartheta \right)$ and  rewrite as follows
\begin{align*}
M_t\left(\vartheta \right)&=\frac{a\sigma ^2}{ P\left(\vartheta
    \right)}\,M_{t-1}\left(\vartheta \right)+\frac{a\Gamma\left(\vartheta \right)}{ P \left(\vartheta \right)}
M_{t-1}\left(\vartheta_0 \right)+\frac{a\Gamma\left(\vartheta \right)\sqrt{ P \left(\vartheta_0 \right)}}{ P \left(\vartheta \right)}\;\zeta _t\left(\vartheta _0\right)\nonumber\\
&=A\left(\vartheta \right)M_{t-1}\left(\vartheta \right)+\left[a-A\left(\vartheta
\right)\right]M_{t-1}\left(\vartheta_0 \right) +B\left(\vartheta ,\vartheta
_0\right)\;\zeta _t\left(\vartheta _0\right),\qquad t\geq 1 .
\end{align*}

The Fisher score  is (see \eqref{10})
\begin{align*}
&\frac{\partial \ln L\left(\vartheta ,X^T\right)}{\partial \vartheta
  }=-\frac{\partial }{\partial \vartheta }
  \sum_{t=1}^{T}\left[\frac{\left(X_t-M_{t-1}\left(\vartheta
      \right)\right)^2}{2 P \left(\vartheta \right)}+\frac{1}{2}\ln  P
    \left(\vartheta \right)\right]\\
 &\qquad \quad
  =\sum_{t=1}^{T}\left[\frac{\left[X_t-M_{t-1}\left(\vartheta
        \right)\right]}{ P \left(\vartheta \right)}\dot
    M_{t-1}\left(\vartheta \right)+ \frac{\left[X_t-M_{t-1}\left(\vartheta
        \right)\right]^2\;\dot  P\left(\vartheta
      \right)}{{2 P\left(\vartheta \right)^2}}-\frac{ \dot
       P\left(\vartheta \right) }{2 P\left(\vartheta \right)}\right]  \\
 &\qquad \quad =\sum_{t=1}^{T}\left[\frac{\left[X_t-M_{t-1}\left(\vartheta
      \right)\right]}{\sqrt{ P\left(\vartheta
    \right)}}\frac{\dot M_{t-1}\left(\vartheta \right)}{\sqrt{ P \left(\vartheta \right)}}+
    \frac{\left[X_t-M_{t-1}\left(\vartheta
        \right)\right]^2}{ P \left(\vartheta \right)}
  \frac{\dot P \left(\vartheta \right)  }{2 P \left(\vartheta \right)}  -\frac{ \dot
       P\left(\vartheta \right) }{2 P\left(\vartheta \right)}\right] \\ &\qquad \quad
  =\frac{1}{\sqrt{ P\left(\vartheta \right)}}\sum_{t=1}^{T}
 \left[ \zeta _t\left(\vartheta \right)\dot M_{t-1}\left(\vartheta
    \right)+\left[\zeta _t\left(\vartheta \right)^2-1 \right]\frac{\dot P \left(\vartheta
    \right)}{2\sqrt{ P \left(\vartheta
    \right)}}\right].
\end{align*}
 Recall that $\dot
  M_{t-1}\left(\vartheta \right)=\partial M_{t-1}\left(\vartheta \right)/\partial
  \vartheta ,\dot
P\left(\vartheta \right)=\partial P\left(\vartheta \right)/\partial
  \vartheta  $ and
\begin{align*}
\zeta_t \left(\vartheta _0\right)=\frac{X_t-M_{t-1}\left(\vartheta _0
  \right)}{{\sqrt{\sigma^2+\Gamma\left(\vartheta _0
      \right)}}}=\frac{X_t-M_{t-1}\left(\vartheta _0 
  \right)   }{\sqrt{ P \left(\vartheta _0\right)}}  ,\qquad t\geq 1  
\end{align*}
 are independent standard Gaussian random variables ( $
\zeta _t\left(\vartheta _0\right)\sim{\cal N}\left(0,1\right)$). 

The equation for derivative  $\dot M_t\left(\vartheta \right),t\geq 1 $ is 
\begin{align*}
\dot M_t\left(\vartheta \right)&=A\left(\vartheta \right)\dot
M_{t-1}\left(\vartheta \right)+\dot A\left(\vartheta
\right)\left[M_{t-1}\left(\vartheta \right)-M_{t-1}\left(\vartheta_0 \right)\right] +\dot B_b\left(\vartheta ,\vartheta
_0\right)\zeta _t\left(\vartheta _0\right)
\end{align*}
with the initial value $\dot M_0\left(\vartheta \right) $. If $\vartheta =\vartheta _0$, then
\begin{align*}
\dot M_t\left(\vartheta _0 \right)&=A\left(\vartheta _0 \right)\dot
M_{t-1}\left(\vartheta _0 \right) +\dot B_b\left(\vartheta_0 ,\vartheta
_0\right)\;\zeta _t\left(\vartheta _0\right),\quad \dot M_0\left(\vartheta_0 \right),\qquad t\geq 1. 
\end{align*}

The stationary version of the process $\dot m_t\left(\vartheta _0
\right),t\geq 1 $ can be written as a sum of i.i.d. variables
\begin{align*}
\dot M_t\left(\vartheta _0 \right)&=\dot B_b\left(\vartheta_0 ,\vartheta
_0\right)\sum_{k=0}^{\infty }A\left(\vartheta _0 \right)^k\zeta
_{t-k}\left(\vartheta _0\right), \\  \dot M_0\left(\vartheta _0
\right)&=\dot B_b\left(\vartheta_0 ,\vartheta 
_0\right)\sum_{k=0}^{\infty }A\left(\vartheta _0 \right)^k\zeta
_{-k}\left(\vartheta _0\right),
\end{align*}
where we introduced i.i.d. r.v.'s $\zeta _k\left(\vartheta _0\right)\sim {\cal
  N}\left(0,1\right), k=0,-1,-2,\ldots$. The real process $\dot
M\left(\vartheta _0 ,t\right),0\leq t\leq T $ has the similar  representation
with the finite sum, but as we are interested by asymptotic ($T\rightarrow
\infty $) properties of estimators we write immediately  this infinite sum and
the difference between these two representations for these processes and for
several other similar processes below are asymptotically
negligible. We have
\begin{align*}
\Ex_{\vartheta _0}\dot M_t\left(\vartheta _0 \right)^2=\dot B_b\left(\vartheta_0
,\vartheta _0\right)^2\sum_{k=0}^{\infty }A\left(\vartheta _0
\right)^{2k}=\frac{\dot B_b\left(\vartheta_0 ,\vartheta
  _0\right)^2}{1-A\left(\vartheta _0 \right)^{2}}=\frac{a^2\sigma ^4\dot P\left(\vartheta _0\right)^2}{P\left(\vartheta _0\right)^3\left(1-A\left(\vartheta _0\right)^2\right)}.
\end{align*}
For the second moment of the score-function we have the following expression 
\begin{align*}
&\Ex_{\vartheta _0}\left[ \frac{\partial \ln L\left(\vartheta
      ,X^T\right)}{\partial \vartheta } \right]^2_{\vartheta =\vartheta    _0}\\ 
&\qquad =\frac{1}{ P\left(\vartheta_0 \right)}\Ex_{\vartheta
    _0}\left(\sum_{t=1}^{T} \left[ \zeta _t\left(\vartheta_0
    \right)\dot M_{t-1}\left(\vartheta_0 \right)+\frac{1}{2}\left[\zeta
      _t\left(\vartheta_0 \right)^2-1 \right]\dot P \left(\vartheta_0
    \right) P \left(\vartheta_0 \right)^{-1/2}\right]\right)^2\\
 &\qquad
  =\frac{1}{ P\left(\vartheta_0 \right)}\Ex_{\vartheta _0}\sum_{t=1}^{T}
  \left( \zeta _t\left(\vartheta_0 \right)\dot M_{t-1}\left(\vartheta_0
  \right)+\frac{1}{2}\left[\zeta _t\left(\vartheta_0 \right)^2-1
    \right]\dot P \left(\vartheta_0 \right) P \left(\vartheta_0
  \right)^{-1/2}\right)^2
\end{align*}
because   (below $t>s$ and for simplicity we omit   $\frac{1}{2} P \left(\vartheta_0
  \right)^{-1/2}   $)
\begin{align*}
&\Ex_{\vartheta _0}\left( \zeta _t\left(\vartheta_0 \right)\dot M_{t-1}\left(\vartheta_0
    \right)+\left[\zeta _t\left(\vartheta_0 \right)^2-1 \right]\dot P \left(\vartheta_0
    \right)\right)\\
&\quad \qquad \qquad\quad \qquad \qquad \times\left( \zeta _s\left(\vartheta_0 \right)\dot M_{s-1}\left(\vartheta_0
    \right)+\left[\zeta _s\left(\vartheta_0 \right)^2-1 \right]\dot P \left(\vartheta_0
    \right)\right)\\
&\qquad \quad =\Ex_{\vartheta _0} \zeta _t\left(\vartheta_0 \right)\dot M_{t-1}\left(\vartheta_0
    \right)\zeta _{s-1}\left(\vartheta_0 \right)\dot M_{s-1}\left(\vartheta_0
    \right)\\
&\qquad \quad\quad \qquad \quad\quad +\dot P \left(\vartheta_0
    \right)\Ex_{\vartheta _0} \zeta _t\left(\vartheta_0 \right)\dot M_{t-1}\left(\vartheta_0
    \right) \left[\zeta _s\left(\vartheta_0 \right)^2-1 \right]\\
&\qquad \quad\quad \qquad \quad\quad +  \dot P \left(\vartheta_0
    \right)  \Ex_{\vartheta _0}\zeta _s\left(\vartheta_0 \right)\dot M_{s-1}\left(\vartheta_0
    \right) \left[\zeta _t\left(\vartheta_0 \right)^2-1 \right]\\
&\qquad \quad\quad \qquad \quad\quad +\dot P \left(\vartheta_0
    \right)^2  \Ex_{\vartheta _0}\left[\zeta _t\left(\vartheta_0
      \right)^2-1 \right]\left[\zeta _s\left(\vartheta_0 \right)^2-1
      \right] =0.
\end{align*}
Here we used the equalities like
\begin{align*}
& \Ex_{\vartheta _0} \left[\zeta _t\left(\vartheta_0 \right)\dot M_{t-1}\left(\vartheta_0
    \right)\zeta _{s}\left(\vartheta_0 \right)\dot M_{s-1}\left(\vartheta_0
    \right)\right]\\
&\qquad \qquad =\Ex_{\vartheta _0} \left[\dot M_{t-1}\left(\vartheta_0
    \right)\zeta _{s}\left(\vartheta_0 \right)\dot M_{s-1}\left(\vartheta_0
    \right)\Ex_{\vartheta _0}\left( \zeta _t\left(\vartheta_0
    \right)|{\goth F}_{t-1}^X  \right)\right]=0.
\end{align*}
Further
\begin{align*}
&\Ex_{\vartheta _0} \left( \zeta _t\left(\vartheta_0 \right)\dot
  M_{t-1}\left(\vartheta_0 \right)+\left[\zeta _t\left(\vartheta_0 \right)^2-1
    \right]\frac{\dot P \left(\vartheta_0 \right)}{2\sqrt{ P \left(\vartheta_0
      \right)}}\right)^2 \\ &\qquad \qquad =\Ex_{\vartheta _0} \zeta
  _t\left(\vartheta_0 \right)^2\dot M_{t-1}\left(\vartheta_0 \right)^2+
  \frac{\dot P \left(\vartheta_0 \right)^2}{4{ P \left(\vartheta_0 \right)}}
  \Ex_{\vartheta _0}\left[\zeta _t\left(\vartheta_0 \right)^2-1
    \right]^2\\ &\qquad \qquad =\Ex_{\vartheta _0} \dot
  M_{t-1}\left(\vartheta_0 \right)^2+ \frac{\dot P \left(\vartheta_0
    \right)^2}{2 P \left(\vartheta_0 \right)} \\ &\qquad \qquad =\frac{\dot
    B_b\left(\vartheta _0,\vartheta _0\right)^2}{1-A\left(\vartheta
    _0\right)^2}+ \frac{\dot P \left(\vartheta_0 \right)^2}{2 P
    \left(\vartheta_0 \right)}\\
&\qquad \qquad =\frac{a^2\sigma ^4\dot P\left(\vartheta
    _0\right)^2}{P\left(\vartheta _0\right)^3\left(1-A\left(\vartheta
    _0\right)^2\right)}+ \frac{\dot P \left(\vartheta_0 \right)^2}{2 P
    \left(\vartheta_0 \right)}\\
&\qquad \qquad =\frac{\dot P\left(\vartheta
    _0\right)^2}{2 P
    \left(\vartheta_0 \right)} \left[\frac{ 2a^2\sigma ^4}{ P\left(\vartheta
      _0\right)^2-a^2\sigma ^4  }+1\right] .
 \end{align*}
Therefore, even if we have a non stationary at the beginning processes the
limit will be  
\begin{align*}
\lim_{T\rightarrow \infty }\frac{1}{T}\Ex_{\vartheta _0}\left[ \frac{\partial \ln
    L\left(\vartheta ,X^T\right)}{\partial \vartheta } \right]^2_{\vartheta
  =\vartheta _0}={\rm I}_b\left(\vartheta _0\right).
\end{align*}

\end{proof}

\begin{remark}
\label{R5}
{\rm If the unknown parameter is $\vartheta =f$ and all other parameters are
  known, then  the
  filtration equations are almost  the same 
\begin{align*}
M_t\left(\vartheta \right)&=a\,M_{t-1}\left(\vartheta
\right)+\frac{a\Gamma\left(\vartheta \right)}{\sigma
  ^2+\Gamma\left(\vartheta \right)}\left[X_t-M_{t-1}\left(\vartheta
  \right)\right],\quad M_t\left(\vartheta \right)=\vartheta m_t\left(\vartheta \right),\quad     t\geq 1,\\ 
\Gamma\left(\vartheta \right)&=\frac{1}{2}\left[{\vartheta^2b^2-\sigma
    ^2\left(1-a^2\right)}\right] +\frac{1}{2}\left[\left({\sigma
  ^2\left(1-a^2\right)}{}-\vartheta ^2b^2\right)^2+{4\vartheta ^2b^2\sigma ^2}\right]^{1/2}.
\end{align*} 

It is easy to see that the score-function has the same form but of course the
derivation now is w.r.t. $\vartheta =f$
\begin{align*}
\frac{\partial L(\vartheta ,X^T)}{\partial \vartheta }=\frac{1}{\sqrt{ P\left(\vartheta \right)}}\sum_{t=1}^{T}
 \left[ \zeta _t\left(\vartheta \right)\dot M_{t-1}\left(\vartheta
    \right)+\left[\zeta _t\left(\vartheta \right)^2-1
      \right]\frac{\dot P_f \left(\vartheta 
    \right)}{2\sqrt{ P \left(\vartheta
    \right)}}\right].
\end{align*}
Here $ P \left(\vartheta \right)=\sigma ^2+\Gamma \left(\vartheta
\right)$.

 To calculate the Fisher information we have to use the following
equation for the derivative $\dot M_{t}\left(\vartheta \right) $
\begin{align*}
\dot M_{t}\left(\vartheta_0 \right)=A\left(\vartheta _0\right)\dot
M_{t-1}\left(\vartheta_0 \right) +\dot B_f\left(\vartheta _0,\vartheta
_0\right)\zeta _t\left(\vartheta _0\right), \qquad t\geq 1
\end{align*}
with the   function 
$$
A\left(\vartheta \right)=\frac{a\sigma ^2}{P \left(\vartheta \right)},\qquad B_f\left(\vartheta ,\vartheta _0\right)=\frac{a\Gamma\left(\vartheta \right)\sqrt{ P \left(\vartheta _0\right)}}{ P \left(\vartheta\right)}.
$$
This gives us the Fisher information
\begin{align}
\label{32}
{\rm I}_f\left(\vartheta _0\right)=\frac{\dot P_f\left(\vartheta_0
\right)^2 \left[P\left(\vartheta_0
\right)^2   +a^2\sigma ^4 \right]  }{2 P
  \left(\vartheta_0\right)^2 \left[P\left(\vartheta_0
\right)^2  -a^2\sigma ^4 \right]}.
\end{align}

}

\end{remark}

\subsection{Unknown parameter $a$}

Suppose that the values  of $f>0,\sigma ^2>0, b>0$ are known and the
parameter $\vartheta =a\in \Theta =\left(\alpha _a,\beta _a\right), -1<\alpha
_a<\beta _a<1$ is unknown. Therefore the  partially observed system is
\begin{align*}
X_t&=f\,Y_{t-1}+\sigma\, w_t,\qquad X_0,\qquad t=1,2,\ldots,\\
Y_t&=\vartheta \,Y_{t-1}+b\, v_t,\qquad \;\;Y_0,
\end{align*}
and the Kalman filter  for $M_t\left(\vartheta \right)=\vartheta m_t\left(\vartheta \right) $ is given by the relations
\begin{align*}
M_t\left(\vartheta \right)&=\left[\vartheta- \frac{\vartheta \Gamma\left(\vartheta \right)}{\sigma
  ^2+\Gamma\left(\vartheta \right)}  \right]  \,M_{t-1}\left(\vartheta
\right)+\frac{\vartheta \Gamma\left(\vartheta \right)}{\sigma
  ^2+\Gamma\left(\vartheta \right)} X_t,\\ 
&= \frac{\vartheta\sigma
  ^2}{\sigma
  ^2+ \Gamma\left(\vartheta \right)}   \,M_{t-1}\left(\vartheta
\right)+\frac{\vartheta \Gamma\left(\vartheta \right)}{\sigma
  ^2+ \Gamma\left(\vartheta \right)} X_t,\\
&= A\left(\vartheta \right)   \,M_{t-1}\left(\vartheta
\right)+\frac{\vartheta \Gamma\left(\vartheta \right)}{P\left(\vartheta
  \right)} M_{t-1}\left(\vartheta _0\right)+B\left(\vartheta ,\vartheta
_0\right)\zeta _t\left(\vartheta _0\right),\quad t\geq 1,\\  
\Gamma\left(\vartheta
\right)&=\frac{1}{2} \left[f^2b^2-\sigma ^2\left(1-\vartheta ^2\right)\right]
+\frac{1}{2}\left[\left({\sigma
    ^2\left(1-\vartheta^2\right)}-f^2b^2\right)^2+4b^2\sigma
    ^2\right]^{1/2}.
 \end{align*}
Here
\begin{align*}
A\left(\vartheta \right)=\frac{\vartheta \sigma ^2}{ P
  \left(\vartheta\right)},\qquad B\left(\vartheta ,\vartheta
_0\right)=\frac{\vartheta \Gamma\left(\vartheta \right)\sqrt{ P
    \left(\vartheta _0\right) } }{ P \left(\vartheta \right)} ,
\end{align*}
and
\begin{align*}
\zeta_t \left(\vartheta _0\right)=\frac{X_t-M_{t-1}\left(\vartheta _0
  \right)}{{\sqrt{\sigma^2+\Gamma\left(\vartheta_0 \right)}}} ,\qquad
t\geq 1
\end{align*}
 are independent standard Gaussian random variables. 

We present the corresponding score-function and Fisher information without
detailed proofs. 

We have 
\begin{align*}
&\frac{\partial \ln L\left(\vartheta ,X^T\right)}{\partial \vartheta
  }  =\frac{1}{\sqrt{ P\left(\vartheta \right)}}\sum_{t=1}^{T}
 \left[ \zeta _t\left(\vartheta \right)\dot M_{t-1}\left(\vartheta
    \right)+\left[\zeta _t\left(\vartheta \right)^2-1 \right]\frac{\dot P \left(\vartheta
    \right)}{2\sqrt{ P \left(\vartheta 
    \right)}}\right].
\end{align*}

The equation for $\dot M_t\left(\vartheta _0\right)=f\partial
m_t\left(\vartheta _0\right)/\partial a $ is 
\begin{align*}
\dot M_t\left(\vartheta _0\right)=A\left(\vartheta_0 \right)\dot M_{t-1}\left(\vartheta
_0\right)+M_{t-1}\left(\vartheta _0\right)+\dot B_a\left(\vartheta
_0,\vartheta _0\right) \zeta _t\left(\vartheta _0\right),\qquad t\geq 1,
\end{align*}
where
\begin{align*}
\dot B_a\left(\vartheta
_0,\vartheta _0\right) =\frac{\Gamma\left(\vartheta_0
  \right)P\left(\vartheta_0 \right)+\vartheta _0\sigma
  ^2\dot P\left(\vartheta_0 \right)}{P\left(\vartheta_0 \right)^{3/2}} . 
\end{align*}
Therefore, using stationarity of all processes we write 
\begin{align*}
\Ex_{\vartheta _0}\dot M_t\left(\vartheta _0\right)^2&=A\left(\vartheta_0
\right)^2\Ex_{\vartheta _0}\dot M_{t-1}\left(\vartheta
_0\right)^2+\Ex_{\vartheta _0}M_{t-1}\left(\vartheta _0\right)^2+\dot
B_a\left(\vartheta _0,\vartheta _0\right)^2\\
&\qquad +2A\left(\vartheta_0
\right)\Ex_{\vartheta _0}\dot M_{t-1}\left(\vartheta
_0\right)M_{t-1}\left(\vartheta _0\right)\\
&=\left(1-A\left(\vartheta_0
\right)^2\right)^{-1}\left[\Ex_{\vartheta _0}M_{t-1}\left(\vartheta _0\right)^2 +\dot
B_a\left(\vartheta _0,\vartheta _0\right)^2\right.\\
&\left.\qquad+2A\left(\vartheta_0
\right)\Ex_{\vartheta _0}\dot M_{t-1}\left(\vartheta
_0\right)M_{t-1}\left(\vartheta _0\right)\right]
\end{align*}
The equations for $M_t\left(\vartheta _0\right)$ and $\dot
M_t\left(\vartheta _0\right) $ allow us to  calculate the following moments
\begin{align*}
\Ex_{\vartheta _0}M_{t-1}\left(\vartheta _0\right)^2&=\frac{\vartheta _0^2 \Gamma 
  \left(\vartheta _0\right)^2}{P \left(\vartheta _0\right)^2  \left(1-\vartheta _0^2\right)},\\
\Ex_{\vartheta _0}\dot M_{t-1}\left(\vartheta
_0\right)M_{t-1}\left(\vartheta _0\right)&=\frac{\vartheta _0\Gamma \left(\vartheta _0\right)^2}{P\left(\vartheta
  _0\right)^2 \left(1-\vartheta
  _0 A\left(\vartheta _0\right)\right)}\left[\frac{\vartheta _0}{\left( 1-A\left(\vartheta
  _0\right)^2\right) }+  P \left(\vartheta
  _0\right)\right.\\
&\qquad \left.+\frac{\vartheta _0\sigma ^2\dot P\left(\vartheta _0\right)
  }{\Gamma \left(\vartheta _0\right)}\right].  
\end{align*}
Hence
\begin{align*}
\Ex_{\vartheta _0}\dot M_t\left(\vartheta
_0\right)^2&=\left(1-A\left(\vartheta _0\right)^2\right)^{-1}\left[\frac{\vartheta _0^2 \Gamma 
  \left(\vartheta _0\right)^2}{P \left(\vartheta _0\right)^2
    \left(1-\vartheta _0^2\right)}+ \frac{\left[\Gamma \left(\vartheta
      _0\right)P\left(\vartheta _0\right)+\sigma ^2\dot P \left(\vartheta
      _0\right) \right]^2}{P \left(\vartheta
      _0\right)^3} \right.\\
  &\left.\quad  +  \frac{2A\left(\vartheta _0\right)  \vartheta _0\Gamma
    \left(\vartheta _0\right)^2}{P\left(\vartheta 
  _0\right)^2 \left(1-\vartheta
  _0 A\left(\vartheta _0\right)\right)}\left(\frac{\vartheta _0}{\left( 1-A\left(\vartheta
  _0\right)^2\right) }+  P \left(\vartheta
  _0\right)+\frac{\vartheta _0\sigma ^2\dot P\left(\vartheta _0\right)
  }{\Gamma \left(\vartheta _0\right)}\right)    \right]\\
&\equiv Q\left(\vartheta _0\right).
\end{align*}
Recall that the Fisher information  is
\begin{align}
\label{33}
{\rm I}_a\left(\vartheta _0\right)&=\frac{1}{T}\Ex_{\vartheta _0}\left(\left.
\frac{\partial \ln L\left(\vartheta ,X^T\right)}{\partial \vartheta 
  }\right|_{\vartheta =\vartheta _0}\right)^2\nonumber\\
&  =\frac{1}{{T P\left(\vartheta \right)}}\sum_{t=1}^{T}
\Ex_{\vartheta _0} \left[ \zeta _t\left(\vartheta \right)\dot M_{t-1}\left(\vartheta
    \right)+\left[\zeta _t\left(\vartheta \right)^2-1 \right]\frac{\dot P \left(\vartheta
    \right)}{2\sqrt{ P \left(\vartheta 
    \right)}}\right]^2\nonumber\\
&  =\frac{1}{{P\left(\vartheta \right)}}\left[\Ex_{\vartheta _0}\dot M_{t-1}\left(\vartheta _0
    \right)^2+\frac{\dot P \left(\vartheta _0
    \right)^2}{2{ P \left(\vartheta  _0\right)}}  \right]\nonumber\\
&=\frac{2Q\left(\vartheta _0\right) +\dot  P \left(\vartheta
  _0\right)^2}{2  P \left(\vartheta _0\right)}.
\end{align}

\begin{remark}
\label{R6}
{\rm The similar calculations allow us to write the score-function and Fisher
  information in the situation where $\vartheta =\sigma ^2$, but we will not
  write it here.
}
\end{remark}

\subsection{Unknown parameter $\vartheta =\left(f,a\right)$}

Consider the system
\begin{align*}
X_t&=\theta_1 \,Y_{t-1}+\sigma\, w_t,\qquad X_0,\qquad t=1,2,\ldots,\\
Y_t&=\theta_2 \,Y_{t-1}+b\, v_t,\qquad \;\;Y_0.
\end{align*}
Here the unknown parameter is $\vartheta =\left(\theta_1,\theta_2\right) $.
Let us denote  $M\left(\vartheta ,t\right)=\theta_1 m_t\left(\vartheta
\right), $ $ \Gamma \left(\vartheta \right)=\theta_1^2\gamma _*\left(\vartheta
\right)$. 
 The Kalman filter we write as follows
\begin{align*}
M\left(\vartheta,t \right)&= \frac{\theta_2\sigma ^2}{{ P}\left(\vartheta
  \right)}M\left(\vartheta,t-1 \right)+\frac{\theta_2
  \Gamma\left(\vartheta \right)}{{ P}\left(\vartheta \right)} X_t\\
&={ A}\left(\vartheta \right)M\left(\vartheta,t-1 \right)+ {\rm
  E}\left(\vartheta \right)M\left(\vartheta_0,t-1 \right)+{\rm
  B}\left(\vartheta,\vartheta _0 \right)\zeta _t\left(\vartheta
_0\right),\qquad \quad t\geq 1,\\  
\Gamma \left(\vartheta \right)&=\frac{1}{2}\left[\theta _1^2b^2-\sigma
  ^2\left(1-\theta _2 ^2\right)\right] +\frac{1}{2}\left[\left(\sigma
  ^2\left(1-\theta_2^2\right)-\theta _1^2b^2\right)^2+4b^2\sigma
  ^2\right]^{1/2},
\end{align*}
where ${ P}\left(\vartheta \right)=\sigma ^2+\Gamma \left(\vartheta \right),$
\begin{align}
\label{34}
 {\rm
  A}\left(\vartheta \right) =\frac{\theta_2\sigma ^2}{{ P}\left(\vartheta
  \right)},\qquad { E}\left(\vartheta \right)=\frac{\theta_2
  \Gamma\left(\vartheta \right)}{{ P}\left(\vartheta \right)},\qquad {\rm
  B}\left(\vartheta ,\vartheta _0\right)= \frac{\theta_2
  \Gamma
\left(\vartheta \right)\sqrt{{ P}\left(\vartheta_0 \right)}}{{
     P}\left(\vartheta \right)}. 
\end{align}
The equations for derivatives $\dot M_f\left(\vartheta _0,t\right),\dot
M_a\left(\vartheta _0,t\right) $ are
\begin{align*}
\dot M_f\left(\vartheta _0,t\right)&={ A}\left(\vartheta_0 \right)\dot
M_f\left(\vartheta_0,t-1 \right)+\dot { 
  B}_f\left(\vartheta_0,\vartheta _0 \right)\zeta _t\left(\vartheta
_0\right),\\
\dot M_a\left(\vartheta _0,t\right)&={ A}\left(\vartheta_0 \right)\dot
M_a\left(\vartheta_0,t-1 \right)+ M\left(\vartheta_0,t-1 \right)+\dot { 
  B}_a\left(\vartheta_0,\vartheta _0 \right)\zeta _t\left(\vartheta
_0\right),
\end{align*}
where recall that
\begin{align*}
M\left(\vartheta_0,t-1 \right)&= \theta_{0,2} M\left(\vartheta_0,t-2 \right)+ {\rm
  B}\left(\vartheta_0,\vartheta _0 \right)\zeta _{t-1}\left(\vartheta
_0\right),\\
\dot { 
  B}_f\left(\vartheta_0,\vartheta _0 \right)&=  \frac{\theta _{0,2}\sigma
  ^2\dot P_f\left(\vartheta _0\right)}{P\left(\vartheta _0\right)^{3/2}}         ,\qquad 
\dot { 
  B}_a\left(\vartheta_0,\vartheta _0 \right)=\frac{\Gamma \left(\vartheta
  _0\right)P\left(\vartheta _0\right)+\theta _{0,2}\dot P_a\left(\vartheta
  _0\right) }{P\left(\vartheta _0\right)^{3/2}}.        
\end{align*}
Here we used the relation ${ A}\left(\vartheta \right)+{\rm
  E}\left(\vartheta \right)=\theta _{2} $.      The vector of score-function is
\begin{align*}
\left.\frac{\partial L(\vartheta ,X^T)}{\partial \theta _1}\right|_{\vartheta
  =\vartheta _0}&=\frac{1}{\sqrt{{ 
      P}\left(\vartheta _0\right)}}\sum_{t=1}^{T}\left[\zeta
  _t\left(\vartheta _0\right)\dot M_f\left(\vartheta
  _0,t-1\right)+\left[\zeta _t\left(\vartheta _0\right)^2-1\right]
  \frac{\dot { P}_f\left(\vartheta _0\right)}{2{ P}\left(\vartheta
    _0\right)}\right],\\
\left.\frac{\partial L(\vartheta ,X^T)}{\partial \theta _2}\right|_{\vartheta
  =\vartheta _0}&=\frac{1}{\sqrt{{ 
      P}\left(\vartheta _0\right)}}\sum_{t=1}^{T}\left[\zeta
  _t\left(\vartheta _0\right)\dot M_a\left(\vartheta
  _0,t-1\right)+\left[\zeta _t\left(\vartheta _0\right)^2-1\right]
  \frac{\dot { P}_a\left(\vartheta _0\right)}{2{ P}\left(\vartheta _0\right)}\right].
\end{align*}
Using the stationarity of the underlying processes we write
\begin{align*}
&\Ex_{\vartheta _0}\left(\left.\frac{\partial L(\vartheta ,X^T)}{\partial
    \theta _1}\right|_{\vartheta =\vartheta _0}\left.\frac{\partial
    L(\vartheta ,X^T)}{\partial \theta _2}\right|_{\vartheta =\vartheta _0}  \right) \\
 &\qquad\qquad\qquad\quad =\frac{T}{{ P}\left(\vartheta
    _0\right)}\left[\Ex_{\vartheta _0} \dot M_f\left(\vartheta
    _0,t-1\right)\dot M_a\left(\vartheta _0,t-1\right)+ \frac{\dot {\rm
        P}_f\left(\vartheta _0\right)\dot { P}_a\left(\vartheta
      _0\right)}{2{ P}\left(\vartheta _0\right)^2} \right].
\end{align*}
Further
\begin{align*}
&\Ex_{\vartheta _0}
  \left[\dot M_a\left(\vartheta _0,t-1\right) \dot M_f\left(\vartheta
    _0,t-1\right)\right]={ A}\left(\vartheta_0 \right)^2\Ex_{\vartheta _0}
  \left[\dot M_a\left(\vartheta _0,t-2\right) \dot M_f\left(\vartheta
    _0,t-2\right)\right]\\
&\qquad \quad +{ A}\left(\vartheta_0 \right)\Ex_{\vartheta _0}
  \left[ M\left(\vartheta _0,t-2\right) \dot M_f\left(\vartheta
    _0,t-2\right)\right]+\dot {\rm
  B}_a\left(\vartheta_0,\vartheta _0 \right)\dot {\rm
  B}_f\left(\vartheta_0,\vartheta _0 \right)\\
&\qquad =\frac{1}{\left(1-{ A}\left(\vartheta_0 \right)^2\right)}\left[
    \frac{{ A}\left(\vartheta_0 \right)\dot {\rm
  B}_f\left(\vartheta_0,\vartheta _0 \right) {\rm
  B}\left(\vartheta_0,\vartheta _0 \right)  }{\left(1-\theta _{0,2}{\rm
      A}\left(\vartheta_0 \right)\right)}+  \dot {\rm
  B}_a\left(\vartheta_0,\vartheta _0 \right)\dot {\rm
  B}_f\left(\vartheta_0,\vartheta _0 \right)  \right] \\
&\qquad =\frac{{ A}\left(\vartheta_0 \right) {\rm
  B}\left(\vartheta_0,\vartheta _0 \right)\dot {\rm
  B}_f\left(\vartheta_0,\vartheta _0 \right)+ \left(1-\theta _{0,2}{\rm
      A}\left(\vartheta_0 \right)\right) \dot {\rm
  B}_a\left(\vartheta_0,\vartheta _0 \right)\dot {\rm
  B}_f\left(\vartheta_0,\vartheta _0 \right)    }{\left(1-\theta _{0,2}{\rm
      A}\left(\vartheta_0 \right)\right)\left(1- { A}\left(\vartheta_0
    \right)^2\right) }\equiv  { K}\left(\vartheta _0\right)
\end{align*}
because
\begin{align*}
\Ex_{\vartheta _0}
  \left[ M\left(\vartheta _0,t-2\right) \dot M_f\left(\vartheta
    _0,t-2\right)\right]=\frac{{\rm
  B}\left(\vartheta_0,\vartheta _0 \right)\dot {\rm
  B}_f\left(\vartheta_0,\vartheta _0 \right)  }{\left(1-\theta _{0,2}{\rm
      A}\left(\vartheta_0 \right)\right)}.
\end{align*}
Hence
\begin{align}
\label{35}
{ I}_{12}\left(\vartheta_0 \right)&=\frac{2{ P}\left(\vartheta
  _0\right)^2 { K}\left(\vartheta _0\right) + \dot { P}_a\left(\vartheta
  _0\right)\dot { P}_f\left(\vartheta _0\right) }{ 2{ P}\left(\vartheta
  _0\right)^3 }.
\end{align}

The Fisher information matrix is
\begin{align*}
{\bf I}\left(\vartheta _0\right)=\begin{pmatrix}
{ I}_{11}\left(\vartheta_0 \right),  &{ I}_{12}\left(\vartheta_0 \right)\\
{ I}_{12}\left(\vartheta_0 \right) &{ I}_{22}\left(\vartheta_0 \right)
\end{pmatrix},
\end{align*}
where the values ${ I}_{11}\left(\vartheta_0 \right) $ and ${\rm
  I}_{22}\left(\vartheta_0 \right) $ are given in \eqref{32} and \eqref{33}
respectively.

\section{One-step MLE-process}

We consider the construction of One-step MLE-process in the case of unknown
parameter $\vartheta =b\in\Theta =\left(\alpha _b,\beta _b\right) $,$\alpha
_b>0$. Let us fix a learning interval of observations
$X^{\tau _\zT }=\left(X_0,X_1,\ldots,X_{\tau _\zT}\right)$, where  $\tau _\zT=\left[T^\delta
  \right], \delta \in \left(\frac{1}{2},1\right)$ and $\left[A\right]$ here is the
integer part of $A$.  As preliminary estimator we take the MME $\vartheta _{\tau _\zT}^*=b_{\tau
  _\zT}^* $ defined in  \eqref{22}
\begin{align*}
\vartheta _{\tau _\zT}^*=2^{-1/2}f^{-1}\left(\left[S_{1,\tau
    _\zT}\left(X^{\tau _\zT}\right)-2\sigma
  ^2\right]\left(1+a\right)\right)^{1/2} .
\end{align*}

One-step MLE-process is 
\begin{align}
\label{36}
&\vartheta _{t,T}^\star=\vartheta _{\tau _\zT}^*+\frac{1}{{\rm
      I}_b(\vartheta _{\tau _\zT}^*)\left(t-\tau _\zT\right) }\sum_{s=\tau
    _\zT+1}^{t}\left[\frac{\left[X_s-fm_{s-1}(\vartheta_{\tau _\zT}^*
        )\right]}{ P (\vartheta_{\tau _\zT}^*)}f\dot
    m_{s-1}(\vartheta_{\tau _\zT}^* )\right.\nonumber\\
 &\qquad \quad \left.+   \left( \left[X_s-fm_{s-1}(\vartheta_{\tau _\zT}^*
        )\right]^2- P(\vartheta_{\tau _\zT}^* )\right)\frac{ \dot
       P(\vartheta_{\tau _\zT}^* )
    }{2 P(\vartheta_{\tau _\zT}^* )^2}\right],\qquad t\in \left[ \tau
    _\zT+2,T\right].
\end{align}

\begin{theorem}
\label{T1}  If $t=\left[vT\right], v\in (0,1]$ and $T\rightarrow \infty $,
  then the following convergences 
\begin{align*}
\sqrt{t}\left(\vartheta _{t,T}^\star-\vartheta _0\right)\Longrightarrow {\cal
  N}\left(0,{\rm I}_b(\vartheta _0)^{-1}\right),\qquad t\Ex_{\vartheta _0}\left(\vartheta _{t,T}^\star-\vartheta _0\right)^2\longrightarrow {\rm I}_b(\vartheta _0)^{-1}.
\end{align*}
hold   uniformly on   compacts $\KK\subset\Theta  $.
\end{theorem}
\begin{proof}
To study this estimate we need the bounds on the first and second derivatives
of $m_t\left(\vartheta \right)$ presented in the next lemma.

\begin{lemma}
\label{L2} 
  For any $p>1$ there exist constants $C_1>0,C_2>0$ not depending on
 $\vartheta_0\in\Theta $ and $t\geq 1$ such that for
\begin{align}
\label{38}
\sup_{\vartheta \in\Theta }\Ex_{\vartheta _0}\left|\dot m_t\left(\vartheta
\right)\right|^p<C_1, \qquad \qquad \sup_{\vartheta \in\Theta }\Ex_{\vartheta _0}\left|\ddot m_t\left(\vartheta
\right)\right|^p<C_2.
\end{align}
\end{lemma}
\begin{proof}
 The equation for the second derivative $\ddot m\left(\vartheta ,t\right)$ is
\begin{align*}
\ddot m_{t}\left(\vartheta \right)&=A\left(\vartheta \right)\ddot
m_{t-1}\left(\vartheta\right)+ 2\dot A\left(\vartheta \right)\dot
m_{t-1}\left(\vartheta\right)+\ddot A\left(\vartheta \right)
m_{t-1}\left(\vartheta\right)\nonumber\\
&\qquad \qquad -\ddot A\left(\vartheta \right)
m_{t-1}\left(\vartheta_0\right) +\ddot B^*\left(\vartheta ,\vartheta
_0\right)\zeta_t\left(\vartheta _0\right) .
\end{align*}
Note that here $B^*\left(\vartheta ,\vartheta _0\right)=af\gamma
_*\left(\vartheta \right)\sqrt{P\left(\vartheta _0\right)}P\left(\vartheta
\right)^{-1} $.

Recall that the stationary versions of the random functions
$m  \left(\vartheta _0 ,t-1\right),$\\ $m\left(\vartheta ,t-1\right)$, $\dot
m\left(\vartheta ,t-1\right)$ and $\ddot
m\left(\vartheta ,t-1\right)$ are
\begin{align*}
&m_{t-1}\left(\vartheta _0\right)=B^*\left(\vartheta _0,\vartheta
_0\right)\sum_{k=0}^{\infty } a^k\zeta_{t-k-1},\\ 
&m_{t-1}\left(\vartheta\right)=A\left(\vartheta
\right)m_{t-2}\left(\vartheta\right) +\left[a-A\left(\vartheta
  \right)\right]B^*\left(\vartheta _0,\vartheta _0\right)\sum_{k=0}^{\infty }
a^k\zeta _{t-k-2}+B^*\left(\vartheta ,\vartheta _0\right)\zeta
_{t-k-1}\\
&\qquad =\left[a-A\left(\vartheta
  \right)\right]B^*\left(\vartheta _0,\vartheta _0\right)\sum_{j=0}^{\infty
}A\left(\vartheta \right)^j   \sum_{k=0}^{\infty }
a^k\zeta _{t-j-k-2}   +  B^*\left(\vartheta ,\vartheta _0\right) \sum_{j=0}^{\infty
}A\left(\vartheta \right)^j \zeta _{t-j-1}      ,                      \\
&\dot m_{t-1}\left(\vartheta \right)=B^*\left(\vartheta_0,\vartheta _0 \right)\dot
A\left(\vartheta \right) \,A\left(\vartheta \right) ^{-1}\sum_{j=0}^{\infty
}\left[aj- A\left(\vartheta \right)\left(j+1\right)\right]A\left(\vartheta
\right)^j \sum_{k=0}^{\infty }
a^k\zeta _{t-j-k-2} \\
&\qquad \quad \qquad \quad  + \sum_{j=0}^{\infty
}\left[\dot B^*\left(\vartheta ,\vartheta _0\right)+jB^*\left(\vartheta ,\vartheta _0\right)\dot
A\left(\vartheta \right) \,A\left(\vartheta \right) ^{-1} \right]
A\left(\vartheta \right)^j \zeta _{t-j-1},\\
&\ddot m_{t}\left(\vartheta \right)= \sum_{j=0}^{\infty
}\left[F_0\left(\vartheta ,\vartheta _0\right)+jF_1\left(\vartheta ,\vartheta
  _0\right)+j^2F_2\left(\vartheta ,\vartheta  _0\right)\right]A\left(\vartheta 
\right)^j \sum_{k=0}^{\infty }a^k\zeta _{t-j-k-2}  \\
&\qquad \quad \qquad \quad +\sum_{j=0}^{\infty
}\left[H_0\left(\vartheta ,\vartheta _0\right)+jH_1\left(\vartheta ,\vartheta
  _0\right)+j^2H_2\left(\vartheta ,\vartheta  _0\right) \right]
A\left(\vartheta \right)^j \zeta _{t-j-1}.
\end{align*}
The bounded functions $F_i\left(\vartheta ,\vartheta  _0\right),H_i\left(\vartheta
,\vartheta  _0\right),i=0,1,2,   $ can be easily calculated by the formal
differentiation of $\dot m_{t-1}\left(\vartheta \right) $.

From these representations, \eqref{27}, condition  $a^2\in\left(0,1\right)$ 
and boundedness of the functions $F_i\left(\cdot ,\cdot \right),H_i\left(\cdot
,\cdot \right)$ it follows that the Gaussian processes $\dot
m_{t-1}\left(\vartheta \right) $ and $\ddot m_{t-1}\left(\vartheta \right) $
have bounded variances and fast decreasing covariance functions.

For example,
\begin{align*}
&\Ex_{\vartheta _0}\left(\sum_{j=0}^{\infty }j^2A\left(\vartheta
  \right)^j\sum_{k=0}^{\infty }a^k\zeta _{t-j-k-2} \right)^2\\
&\qquad \quad =\Ex_{\vartheta _0}\left(\sum_{j=0}^{\infty }j^2A\left(\vartheta
  \right)^j\sum_{k=0}^{\infty }a^k\zeta _{t-j-k-2} \sum_{l=0}^{\infty }l^2A\left(\vartheta
  \right)^l\sum_{m=0}^{\infty }a^m\zeta _{t-m-l-2} \right)\\
&\qquad \quad =\sum_{j=0}^{\infty }j^2A\left(\vartheta
  \right)^j\sum_{l=0}^{\infty }l^2A\left(\vartheta
  \right)^l\sum_{k=0}^{\infty }\sum_{m=0}^{\infty }a^ka^m\Ex_{\vartheta
    _0}\Bigl( \zeta _{t-j-k-2} \zeta _{t-m-l-2} \Bigr)\\ 
&\qquad \quad =\sum_{j=0}^{\infty }j^2A\left(\vartheta
  \right)^j\sum_{l=0}^{\infty }l^2A\left(\vartheta
  \right)^la^{j-l}\sum_{k\geq l-j\geq 0}^{\infty }a^{2k}\\
&\qquad \quad =\frac{1}{1-a^2}\sum_{j=0}^{\infty }j^2A\left(\vartheta
  \right)^j\sum_{l\geq j}^{\infty }l^2A\left(\vartheta
  \right)^la^{j-l}a^{2\left(l-j\right)}\\
&\qquad \quad =\frac{1}{1-a^2}\sum_{j=0}^{\infty }j^2A\left(\vartheta
  \right)^ja^{-j}\sum_{l\geq j}^{\infty }l^2\left[aA\left(\vartheta
  \right)\right]^l.
\end{align*}
Recall that
\begin{align*}
 \sum_{l\geq j}^{\infty } l^2\left[aA\left(\vartheta
   \right)\right]^{l}&=\left[\ln aA\left(\vartheta
   \right)\right]^{-2}\frac{\partial ^2}{\partial x^2}\sum_{l\geq j}^{\infty
 }\left.\left[aA\left(\vartheta \right)\right]^{xl}\right|_{x=1}\\
&= \left[\ln aA\left(\vartheta \right)\right]^{-2}\frac{\partial ^2}{\partial
   x^2}\left.\left(\frac{\left[aA\left(\vartheta
     \right)\right]^{xj}}{1-\left[aA\left(\vartheta
     \right)\right]^{x}}\right)\right|_{x=1}.
\end{align*}

 Therefore the moments $\Ex_{\vartheta _0}\dot m_t\left(\vartheta \right)^2 $
 and $\Ex_{\vartheta _0}\ddot m_t\left(\vartheta \right)^2 $ can be calculated
 exactly and the constants $C_1,C_2$ in \eqref{38} can be   chosen not depending on
 $\vartheta_0\in\Theta $ and $t\geq 0$.

\end{proof}

Consider the normalized difference
\begin{align*}
&\sqrt{T}\left(\vartheta _{t,T}^\star-\vartheta
  _0\right)=\sqrt{T}\left(\vartheta _{\tau _\zT}^*-\vartheta
  _0\right)\nonumber\\
&\qquad +\frac{\sqrt{T}}{{\rm 
      I}_b(\vartheta _{\tau _\zT}^*)\left(t-\tau _\zT\right) }\sum_{s=\tau
    _\zT+1}^{t}\left[\frac{\left[X_s-fm_{s-1}(\vartheta_{\tau _\zT}^*
        )\right]}{ P (\vartheta_{\tau _\zT}^*)}f\dot
    m_{s-1}(\vartheta_{\tau _\zT}^* )\right.\nonumber\\
 &\qquad \quad \left.+    \left(\left[X_s-fm_{s-1}(\vartheta_{\tau _\zT}^*
        )\right]^2-  P(\vartheta_{\tau _\zT}^*
      )\right)\frac{ \dot
       P(\vartheta_{\tau _\zT}^* )
    }{2 P(\vartheta_{\tau _\zT}^* )^2}\right],\qquad t\in \left[ \tau
    _\zT+2,T\right].
\end{align*}
We can write
\begin{align*}
&X_s-fm_{s-1}(\vartheta_{\tau _\zT}^* )=X_s-fm_{s-1}(\vartheta_0
)+f\left[m_{s-1}(\vartheta_0 ) -m_{s-1}(\vartheta_{\tau _\zT}^* )\right]\nonumber
\\ &\qquad \;\;\; =\sqrt{ P \left(\vartheta _0\right)}\,\zeta _s\left(\vartheta
_0\right)- (\vartheta_{\tau _\zT}^*-\vartheta _0 )f\dot m_{s-1}\left(\vartheta
_0\right)-\frac{f}{2}(\vartheta_{\tau _\zT}^*-\vartheta _0 )^2\ddot
m_{s-1}(\tilde\vartheta ),\nonumber\\
& P (\vartheta_{\tau _\zT}^*)= P (\vartheta_0)+(\vartheta_{\tau
  _\zT}^*-\vartheta _0  )\dot  P (\vartheta_0)+\frac{1}{2}(\vartheta_{\tau
  _\zT}^*-\vartheta _0  )^2\ddot  P (\tilde\vartheta), \nonumber\\
& {\rm I}_b(\vartheta _{\tau _\zT}^*)={\rm I}_b(\vartheta _0)+O\left(\vartheta_{\tau
  _\zT}^*-\vartheta _0  \right)\nonumber\\
 &\left[X_s-fm_{s-1}(\vartheta_{\tau _\zT}^*
  )\right]^2-  P (\vartheta_{\tau _\zT}^*  )  = P (\vartheta_0  )\left[\zeta
_s\left(\vartheta _0\right)^2-1 \right] -(\vartheta_{\tau _\zT}^*-\vartheta _0
) \dot P \left(\vartheta _0\right)\nonumber\\
&\qquad \qquad -2
(\vartheta_{\tau _\zT}^*-\vartheta _0 )\sqrt{ P (\vartheta_0  )}\zeta
_s\left(\vartheta _0\right)  \dot m_{s-1}\left(\vartheta
 _0\right)+O\left( (\vartheta_{\tau _\zT}^*-\vartheta _0 )^2 \right).
\end{align*}
The substitution of these relations in \eqref{36} yields 
\begin{align*}
&\vartheta _{t,T}^\star =\vartheta_{\tau _\zT}^*\nonumber\\
&\quad + \sum_{s=\tau _\zT+1}^{t}\frac{\left[2\zeta
      _{s}(\vartheta_{0} )f\dot m_{s-1}(\vartheta_{0}
      )\sqrt{ P(\vartheta_0 )}+ \left[\zeta _{s}(\vartheta_{0}
        )^2-1\right]\dot  P(\vartheta_0 )\right] }{2\left(t-\tau
    _\zT\right) {\rm I}_b(\vartheta _0) P (\vartheta_0)
  }\left(1+O\left(\vartheta_{\tau _\zT}^*-\vartheta _0 \right)\right)\nonumber\\
 &\qquad  - \frac{(\vartheta_{\tau _\zT}^*-\vartheta _0 )}{\left(t-\tau
    _\zT\right) {\rm I}_b(\vartheta _0)} \sum_{s=\tau
    _\zT+1}^{t}\left(\frac{f^2\dot m_{s-1}\left(\vartheta _0\right)^2}{ P
    (\vartheta_0) }+\frac{ \dot P
    (\vartheta_0)^2 }{2 P
    (\vartheta_0)^2 }\right)+O\left( (\vartheta_{\tau _\zT}^*-\vartheta _0 )^2
  \right).
\end{align*}
Note that if $t\rightarrow \infty $ then  by LLN and by the CLT uniformly on
compacts $\KK\subset \Theta $ we have
\begin{align*}
R_{1,T}&=\frac{1}{t}\sum_{s=\tau
    _\zT+1}^{t}\left(\frac{f^2\dot m_{s-1}\left(\vartheta _0\right)^2}{ P
    (\vartheta_0) }+\frac{ \dot P
    (\vartheta_0)^2 }{2 P
    (\vartheta_0)^2 }\right)\longrightarrow {\rm I}_b(\vartheta _0),\\
R_{2,T}&=\frac{1}{\sqrt{t}}\sum_{s=\tau
    _\zT+1}^{t}\left({\rm I}_b(\vartheta _0)-\frac{f^2\dot m_{s-1}\left(\vartheta _0\right)^2}{ P
    (\vartheta_0) }-\frac{ \dot P
    (\vartheta_0)^2 }{2 P
    (\vartheta_0)^2 }\right)\Longrightarrow {\cal
  N}\left(0,{\rm D}(\vartheta _0)^2\right),\\
R_{3,T}&=\frac{1}{t}\Ex_{\vartheta _0}\left[\sum_{s=\tau
    _\zT+1}^{t}\left(\frac{f^2\dot m_{s-1}\left(\vartheta _0\right)^2}{ P
    (\vartheta_0) }+\frac{ \dot P
    (\vartheta_0)^2 }{2 P
    (\vartheta_0)^2 }-{\rm I}_b(\vartheta _0)\right)\right]^2\\
&=\frac{1}{tP
    (\vartheta_0)^2}\Ex_{\vartheta _0}\left[\sum_{s=\tau
    _\zT+1}^{t}\Bigl({f^2\left[\dot m_{s-1}\left(\vartheta
    _0\right)^2-\Ex_{\vartheta _0} \dot m_{s-1}\left(\vartheta
    _0, t\right)^2 \right]}{  }\Bigr)\right]^2\leq C,\\
  \end{align*}
where the constant $C>0$ does not depend on $t$ and $\vartheta _0\in\KK$. 
We have as well the uniform convergence
\begin{align*}
\frac{1}{\sqrt{t}}\sum_{s=\tau _\zT+1}^{t}\frac{\left[2\zeta
      _{s}(\vartheta_{0} )f\dot m_{s-1}(\vartheta_{0}
      )\sqrt{ P(\vartheta_0 )}+ \left[\zeta _{s}(\vartheta_{0}
        )^2-1\right]\dot  P(\vartheta_0 )\right] }{2\left(t-\tau
    _\zT\right) {\rm I}_b(\vartheta _0) P (\vartheta_0)}\Longrightarrow {\cal
  N}\left(0,{\rm I}_b(\vartheta _0)^{-1}\right).
\end{align*}
The value  $\sup_{\vartheta _0\in\Theta }{\rm D}(\vartheta _0)^2<\infty  $ can
be calculated too, but we need not it. 
Therefore if we put $t=vT$, $v\in (T^{\delta -1},1]$, then
\begin{align*}
&\sqrt{t}\left(\vartheta _{t,T}^\star-\vartheta _0\right)
=\sqrt{t}\left(\vartheta_{\tau _\zT}^*-\vartheta _0\right)\\
&\quad + \sum_{s=\tau _\zT+1}^{t}\frac{\left[2\zeta
      _{s}(\vartheta_{0} )f\dot m_{s-1}(\vartheta_{0}
      )\sqrt{ P(\vartheta_0 )}+ \left[\zeta _{s}(\vartheta_{0}
        )^2-1\right]\dot  P(\vartheta_0 )\right] }{2\left(\sqrt{t}-\tau
    _\zT\right) {\rm I}_b(\vartheta _0) P (\vartheta_0)
  }\left(1+O\left(\vartheta_{\tau _\zT}^*-\vartheta _0 \right)\right)\\
 &\qquad  - \frac{\sqrt{t}(\vartheta_{\tau _\zT}^*-\vartheta _0 )}{\left(t-\tau
    _\zT\right) {\rm I}_b(\vartheta _0)} \sum_{s=\tau
    _\zT+1}^{t}\left(\frac{f^2\dot m_{s-1}\left(\vartheta _0\right)^2}{ P
    (\vartheta_0) }+\frac{ \dot P
    (\vartheta_0)^2 }{2 P
    (\vartheta_0)^2 }\right)+O\left(\sqrt{t} (\vartheta_{\tau _\zT}^*-\vartheta _0 )^2 \right). 
\end{align*}

Remind that $\delta \in \left(\frac{1}{2},1\right)$ and  the MME
$\vartheta_{\tau _\zT}^* $ satisfies \eqref{22}. Hence, 
for any fixed $v\in (0,1]$
\begin{align*}
\sqrt{t}\Ex_{\vartheta _0}\left|\vartheta_{\tau _\zT}^*-\vartheta
_0\right|^2\leq \frac{C\sqrt{t}}{\tau _\zT^2}\leq
\frac{C\,\sqrt{v}T^{1/2}}{T^{2\delta }}\longrightarrow 0. 
\end{align*}
Further,
\begin{align*}
&\frac{\sqrt{t}\left(\vartheta_{\tau _\zT}^*-\vartheta _0\right)}{ {\rm
      I}_b(\vartheta _0) }\left[ {\rm I}_b(\vartheta _0)- \frac{1}{t}\sum_{s=\tau
      _\zT+1}^{t}\left(\frac{f^2\dot m_{s-1}\left(\vartheta
      _0\right)^2}{ P (\vartheta_0) }+\frac{ \dot P (\vartheta_0)^2
    }{2 P (\vartheta_0)^2 }\right) \right]\\
 &\qquad \quad
  =\frac{\sqrt{t}\left(\vartheta_{\tau _\zT}^*-\vartheta _0\right)}{ {\rm
      I}_b(\vartheta _0) }\frac{1}{t}\sum_{s=\tau _\zT+1}^{t}\left[ \frac{f^2\dot
      B^*\left(\vartheta _0,\vartheta _0\right)^2}{\left[1-A\left(\vartheta
        _0\right)^2 \right] P (\vartheta_0) }  -  \frac{f^2\dot
    m_{s-1}\left(\vartheta _0\right)^2}{ P (\vartheta_0) } \right]\\
&\qquad \quad
  =\frac{\sqrt{t}\left(\vartheta_{\tau _\zT}^*-\vartheta _0\right)f^2}{ {\rm
      I}_b(\vartheta _0)P (\vartheta_0) } \frac{1}{t}\sum_{s=\tau _\zT+1}^{t}\left[ \frac{\dot
      B^*\left(\vartheta _0,\vartheta _0\right)^2}{\left[1-A\left(\vartheta
        _0\right)^2 \right]  }  -  {\dot
    m_{s-1}\left(\vartheta _0\right)^2}{  } \right]\\
 &\qquad \quad \longrightarrow 0.
\end{align*}
Therefore for any fixed $v\in \left(0,1\right)$ and $t=vT$ uniformly on $\vartheta _0\in\KK$
\begin{align*}
\sqrt{t}
\left(\vartheta _{t,T}^\star-\vartheta _0 \right)\Longrightarrow
            {\cal N}\left(0, {\rm I}_b(\vartheta _0) ^{-1}\right),
\end{align*}
and 
\begin{align*}
{t}
\Ex_{\vartheta _0}\left(\vartheta _{t,T}^\star-\vartheta _0 \right)^2\longrightarrow
             {\rm I}_b(\vartheta _0) ^{-1}.
\end{align*}
\end{proof}

\begin{remark}
\label{R7}
{\rm The One-step MLE-processes in the cases $\vartheta =f $, $\vartheta =a
  $, $\vartheta =\sigma ^2 $ can be constructed following the same lines. We do
  not present the corresponding calculations because it will be mainly
  repetition of the given above proof. 

If the unknown parameter is two-dimensional, say, $\vartheta
=\left(f,a\right)^\top=\left(\theta _1,\theta _2\right)^\top$, then two-dimensional
One-step MLE-process $\vartheta _{t,T}^\star=\left(\theta _{1,t,T}^\star,
\theta _{2,t,T}^\star\right)^\top$ can be constructed as follows. Introduce MMEs
$\theta _{\tau _\zT}^*=(\theta _{1,\tau _\zT}^*,\theta _{2,\tau
  _\zT}^*)^\top$, $\theta _{1,\tau _\zT}^* =f_{\tau _\zT}^* $ and $\theta _{2,\tau _\zT}^*
=a_{\tau _\zT}^* $, where $\tau _\zT=\left[T^\delta \right], \delta  \in
(1/2,1]$ and 
\begin{align*}
\theta _{1,\tau _\zT}^* &=2^{-1/2}b^{-1}\left[{S_{1,{\tau _\zT}}\left(X^{\tau _\zT}\right)(1+\theta
  _{2,\tau _\zT}^*)-2\sigma ^2 }\right]^{1/2} ,\\
 \theta _{2,\tau  _\zT}^* &=\frac{S_{1,{\tau _\zT}}\left(X^{\tau
     _\zT}\right)+S_{2,{\tau _\zT}}\left(X^{\tau _\zT}\right)-\sigma 
  ^2}{S_{1,{\tau _\zT}}\left(X^{\tau _\zT}\right)-2\sigma ^2} .
\end{align*}
Denote 
\begin{align*}
M\left(\vartheta,t \right)=\theta _1m_t\left(\vartheta \right),\; \dot
{\rm M}\left(\vartheta,t \right)= \left( \dot M_f\left(\vartheta,t \right),\dot
M_a\left(\vartheta,t \right)\right) ^\top ,\; \dot{\rm P}\left(\vartheta
\right)=\left(\dot{\rm P}_f\left(\vartheta \right),\dot{\rm
  P}_a\left(\vartheta \right) \right)^\top
.
\end{align*}
The equations for $\dot M_f\left(\vartheta,t \right)$ and $\dot
M_a\left(\vartheta,t \right)$ we obtain by differentiation of the equation 
\begin{align*}
M\left(\vartheta,t \right)&={\rm A}\left(\vartheta \right)M\left(\vartheta,t-1 \right)+ {\rm
  E}\left(\vartheta \right)M\left(\vartheta_0,t-1 \right)+{\rm
  B}\left(\vartheta,\vartheta _0 \right)\zeta _t\left(\vartheta
_0\right),\qquad t\geq 1,
\end{align*}
by $\theta _1$ and $\theta _2$ correspondingly. For definition of the functions ${\rm A}\left(\cdot  \right), {\rm
  E}\left(\cdot \right)$ and ${\rm
  B}\left(\cdot,\cdot\right)$ see \eqref{34}
 Recall that the Fisher information
matrix ${\bf I}\left(\vartheta \right)$ is defined by the relations
\eqref{32},\eqref{33} and \eqref{35}. Suppose that this matrix is
uniformly on $\vartheta \in\Theta $ non degenerate.  Then the  One-step MLE-process is given  by
the equality
\begin{align}
\label{41}
\vartheta _{t,T}^\star&=\theta _{\tau _\zT}^*+\frac{{\bf I}(\theta _{\tau
    _\zT}^* )^{-1}}{t-\tau _\zT} \sum_{s=\tau
    _\zT+1}^{t}\left[\frac{\left[X_s-M(\vartheta_{\tau _\zT}^*
      ,s-1  )\right]}{ P (\vartheta_{\tau _\zT}^*)}\dot
    {\rm M}(\vartheta_{\tau _\zT}^*,{s-1} )\right.\nonumber\\
 & \quad \left.+   \left( \left[X_s-M(\vartheta_{\tau _\zT}^*,{s-1}
        )\right]^2- P(\vartheta_{\tau _\zT}^* )\right)\frac{ \dot
      {\rm P}(\vartheta_{\tau _\zT}^* )
    }{2 P(\vartheta_{\tau _\zT}^* )^2}\right],\qquad t\in \left[ \tau
    _\zT+2,T\right].
\end{align}
It can be shown that for any $v\in (0,T]$, if we put $t=vT$ then  the normalized
  difference is asymptotically normal:
\begin{align*}
&\sqrt{t}\left(\vartheta _{t,T}^\star-\vartheta _0\right)\Longrightarrow \xi\left(\vartheta _0\right)
\sim {\cal
  N}\left(0,{\bf I}(\theta _0 )^{-1} \right),\nonumber\\
&{t}\Ex_{\vartheta _0}  \left\|\vartheta _{t,T}^\star-\vartheta
_0\right\|^2\longrightarrow \Ex_{\vartheta _0}\left\|\xi \left(\vartheta _0\right)\right\|^2.
\end{align*}

}
\end{remark}

\section{MLE and BE}

Consider the model of observations \eqref{1}-\eqref{2}, where the
unknown parameter is $\vartheta =b\in\Theta =\left(\alpha _b,\beta
_b\right),\alpha _b>0$. Below we study the MLE $\hat\vartheta_\zT $ and BE
$\tilde\vartheta_\zT $ defined by the usual relations
\begin{align*}
L(\hat\vartheta_\zT,X^T)=\sup_{\vartheta \in\Theta
}L(\hat\vartheta,X^T),\qquad \qquad
\tilde\vartheta_\zT=\frac{\int_{\Theta }^{}\vartheta p\left(\vartheta
  \right)L(\hat\vartheta,X^T){\rm d}\vartheta }{\int_{\Theta }^{}
  p\left(\vartheta \right)L(\hat\vartheta,X^T){\rm d}\vartheta}.
\end{align*}
Here $p\left(\vartheta \right),\vartheta \in\Theta$ is continuous, positive on
$\Theta $ density a priory. 

Recall the notation
\begin{align*}
 P\left(\vartheta \right)& =\sigma ^2+f^2\gamma _*\left(\vartheta
  \right),\qquad \qquad   A\left(\vartheta \right)=\frac{a\sigma ^2}{
    P\left(\vartheta \right)^2} \\  
   B^*\left(\vartheta,\vartheta_0\right)&=\frac{af\gamma
      _*\left(\vartheta \right)\sqrt{ P \left(\vartheta_0 \right)}}{ P
    \left(\vartheta \right)},\qquad 
 \dot  B^*\left(\vartheta_0,\vartheta_0\right)=\left.\frac{\partial
    B^*\left(\vartheta,\vartheta_0\right)}{\partial \vartheta
  }\right|_{\vartheta =\vartheta _0}= \frac{af\sigma ^2\dot\gamma
   _*\left(\vartheta_0 \right)}{ P 
      \left(\vartheta_0\right)^{3/2} }.\\
{\rm I}_b\left(\vartheta _0\right)&=\frac{f^2\dot B^*\left(\vartheta
    _0,\vartheta _0\right)^2}{\left[1-A\left(\vartheta_0
     \right)^2\right] P\left(\vartheta_0 \right)}+\frac{\dot
     P\left(\vartheta_0 \right)^2}{2 P\left(\vartheta_0
  \right)^2}\\&=\frac{f^2\dot \gamma _*\left(\vartheta \right)^2}{ P
  \left(\vartheta _0\right)}\left[ \frac{a^2\sigma ^4}{ P
  \left(\vartheta _0\right)^2-a^2\sigma ^4}+\frac{1}{2 P
  \left(\vartheta _0\right)}\right].
\end{align*}

Below it is shown (see \eqref{45}-\eqref{46}) that the family of measures is
LAN. Therefore we have Hajek-Le Cam's lower bound: for any estimator
$\bar\vartheta _T$ and any $\vartheta _0\in\Theta $
\begin{align*}
\lim_{\nu \rightarrow 0}\Liminf_{T\rightarrow \infty }\sup_{\left|\vartheta -\vartheta _0\right|\leq \nu }T\Ex_\vartheta
\left(\bar\vartheta _T-\vartheta \right)^2\geq {\rm I}_b\left(\vartheta
_0\right)^{-1} . 
\end{align*}
This bound allows us to give the following definition

Define the asymptotically efficient estimator $\vartheta _T^\circ$ as
estimator   for which the equality
\begin{align*}
\lim_{\nu \rightarrow 0}\lim_{T\rightarrow \infty }\sup_{\left|\vartheta -\vartheta _0\right|\leq \nu }T\Ex_\vartheta
\left(\vartheta _T^\circ-\vartheta \right)^2= {\rm I}_b\left(\vartheta
_0\right)^{-1} . 
\end{align*}
holds for all $\vartheta _0\in\Theta $.

\begin{theorem}
\label{T2} 
  The MLE and BE are
consistent, asymptotically normal,
\begin{align*}
\sqrt{T}\left(\hat\vartheta_\zT-\vartheta _0\right)\Longrightarrow {\cal
  N}\left(0,{\rm I}_b\left(\vartheta _0\right)^{-1}\right),\qquad \qquad
\sqrt{T}\left(\tilde\vartheta_\zT-\vartheta _0\right)\Longrightarrow {\cal
  N}\left(0,{\rm I}_b\left(\vartheta _0\right)^{-1}\right),
\end{align*}
 the polynomial  moments converge and the
both estimators are asymptotically efficient.
\end{theorem}
\begin{proof}
The mentioned properties of the estimators we will proof with the help of the
general Theorems 3.1.1 and 3.2.1 in \cite{IH81}, i.e., we verify the
conditions of these theorems for our model of observations.  Introduce the normalized  ($\varphi _T=T^{-1/2}$) likelihood ratio function
\begin{align*}
Z_T\left(u\right)=\frac{L(\vartheta _0+\varphi _Tu,X^T)}{L(\vartheta
  _0,X^T)},\qquad u\in\UU_T=\left( \sqrt{T}\left(\alpha _T-\vartheta
_0\right),\sqrt{T}\left(\beta _T-\vartheta _0\right)\right)
\end{align*}

Therefore we have to prove the following properties of the process
$Z_T\left(u\right),u\in\UU_T $. 
{\it 
\begin{enumerate}
\item The process $Z_T\left(\cdot \right)$ admits the representation
\begin{align}
\label{45}
Z_T\left(u \right)=\exp\left(u\Delta_\zT (\vartheta _0,X^T)-\frac{u^2}{2}{\rm
  I}_b\left(\vartheta _0\right)+r_\zT(\vartheta _0,u,X^T)\right) ,\qquad u\in\UU_T,
\end{align}
where  uniformly on compacts $\KK\subset\Theta $ the
convergence   $r_\zT(\vartheta _0,u,X^T)\rightarrow 0$ holds,
\begin{align*}
\Delta_\zT (\vartheta _0,X^T)=\frac{1}{\sqrt{T P\left(\vartheta_0
    \right)}}\sum_{t=1}^{T} \left[ \zeta _t\left(\vartheta_0
  \right)f\dot m_{t-1}\left(\vartheta_0 \right)+\left[\zeta
    _t\left(\vartheta_0 \right)^2-1 \right]\frac{\dot P \left(\vartheta_0
    \right)}{2\sqrt{ P \left(\vartheta_0 \right)}}\right]
\end{align*}
and uniformly on $\KK$
\begin{align}
\label{46}
\Delta_\zT (\vartheta _0,X^T)\Longrightarrow  {\cal
  N}\left(0,{\rm I}_b\left(\vartheta _0\right)\right).
\end{align}

\item  There exists constant $C>0$ such that
\begin{align}
\label{47}
\sup_{\vartheta _0\in\KK}\Ex_{\vartheta _0}\left|Z_T\left(u_2 \right)^{1/2}-Z_T\left(u_1 \right)^{1/2}
\right|^2\leq C\,\left|u_1-u_2\right|^2 .
\end{align}

\item There exists a constant $\kappa >0$ and for any $N>0$ there is a
  constant $C>0$ such that
\begin{align}
\label{48}
\sup_{\vartheta _0\in\KK}\Pb_{\vartheta _0}\left(Z_T\left(u \right)>e^{-\kappa u^2}\right)\leq
\frac{C}{\left|u\right|^N}. 
\end{align}
\end{enumerate}
}

Let us verify \eqref{45}-\eqref{46}. Using Taylor expansions  we can write
(below $\vartheta _u=\vartheta _0+\varphi _T u$)
\begin{align*}
\ln
Z_T\left(u\right)&=-\frac{1}{2}\sum_{t=1}^{T}\left[
\frac{\left(X_t-fm_{t-1}\left(\vartheta_u 
    \right)\right)^2}{{  P \left(\vartheta_u \right) }}+\ln \frac{ P
    \left(\vartheta_u \right)}{ P \left(\vartheta_0 \right) }-
  \frac{\left(X_t-fm_{t-1}\left(\vartheta_0 \right)\right)^2}{{  P
      \left(\vartheta_0 \right) }} \right]\\
&=\frac{u}{\sqrt{T}}\sum_{t=1}^{T}\left[\frac{\zeta _t\left(\vartheta
    _0\right)f\dot m_{t-1}\left(\vartheta _0\right)}{\sqrt{ P
      \left(\vartheta _0\right)}}+ \frac{\left(\zeta _t\left(\vartheta
    _0\right)^2-1\right)\dot P \left(\vartheta _0\right)}{2 P
    \left(\vartheta _0\right)}\right]\\
&\qquad -\frac{u^2}{2T}\left[\frac{f^2\dot m_{t-1}\left(\vartheta
    _0\right)^2}{ P \left(\vartheta _0\right)}+ \frac{\zeta _t\left(\vartheta
    _0\right)^2\dot P \left(\vartheta _0\right)^2}{ P \left(\vartheta
    _0\right)^2}-\frac{\dot P \left(\vartheta _0\right)^2}{2 P
    \left(\vartheta _0\right)^2}\right] +O\left(T^{-1/2}\right).
\end{align*}
Now the representation \eqref{45}-\eqref{46} follows from the CLT and the LLN:
\begin{align*}
&\frac{1}{\sqrt{T}}\sum_{t=1}^{T}\left[\frac{\zeta _t\left(\vartheta
    _0\right)f\dot m_{t-1}\left(\vartheta _0\right)}{\sqrt{ P
      \left(\vartheta _0\right)}}+ \frac{\left(\zeta _t\left(\vartheta
    _0\right)^2-1\right)\dot P \left(\vartheta _0\right)}{2 P
    \left(\vartheta _0\right)}\right]\Longrightarrow  {\cal
  N}\left(0,{\rm I}_b\left(\vartheta _0\right)\right),\\
&\frac{1}{T}\sum_{t=1}^{T}\left[\frac{f^2\dot m_{t-1}\left(\vartheta
    _0\right)^2}{ P \left(\vartheta _0\right)}+ \frac{\zeta _t\left(\vartheta
    _0\right)^2\dot P \left(\vartheta _0\right)^2}{ P \left(\vartheta
    _0\right)^2}-\frac{\dot P \left(\vartheta _0\right)^2}{2 P
    \left(\vartheta _0\right)^2}\right]\longrightarrow {\rm I}_b\left(\vartheta _0\right).
\end{align*}

 Let us denote 
\begin{align*}
\Pi \left(u\right)=-\frac{1}{2}\sum_{t=1}^{T}\left[
\frac{\left(X_t-fm_{t-1}\left(\vartheta_u 
    \right)\right)^2}{{  P \left(\vartheta_u \right) }}+\ln \frac{ P
    \left(\vartheta_u \right)}{ P \left(\vartheta_0 \right) }-
  \frac{\left(X_t-fm_{t-1}\left(\vartheta_0 \right)\right)^2}{{  P
      \left(\vartheta_0 \right) }} \right]
\end{align*}
and calculate its derivative
\begin{align*}
\dot\Pi \left(u\right)&=\frac{\partial }{\partial
    u}\Pi \left(u\right)  = \frac{\varphi _T}{2}\sum_{t=1}^{T}\left[
  \frac{2\left(X_t-fm_{t-1}\left(\vartheta_u \right)\right)f\dot
    m_{t-1}\left(\vartheta_u \right) }{{  P \left(\vartheta_u \right)
  }}-\frac{\dot P \left(\vartheta_u \right)}{ P \left(\vartheta_u
    \right) }\right.\\
&\quad\qquad \qquad \qquad  \left.
   +\frac{\left(X_t-fm_{t-1}\left(\vartheta_u \right)\right)^2\dot P
    \left(\vartheta_u \right)}{ P \left(\vartheta_u \right)^2}\right].
\end{align*}
Note that $X_t-fm_{t-1}\left(\vartheta _u\right)=\zeta _t\left(\vartheta
_u\right)\sqrt{ P \left(\vartheta _u\right)},t\geq 1 $, where $\zeta
_t\left(\vartheta _u\right),t\geq 1$ under measure $\Pb_{\vartheta _u}$ are
i.i.d.  r.v.'s $\zeta _t\left(\vartheta _u\right)\sim {\cal
  N}\left(0,1\right)$. Therefore, with $\Pb_{\vartheta _u} $ probability 1 we have
\begin{align*}
\dot \Pi \left(u\right)&=\varphi _T \sum_{t=1}^{T}\left[
  \frac{\zeta _t\left(\vartheta _u\right)f\dot m_{t-1}\left(\vartheta_u
    \right) }{{ \sqrt{ P \left(\vartheta_u \right)} }}
  +\frac{\left[\zeta _t\left(\vartheta _u\right)^2-1\right]\dot P \left(\vartheta_u
    \right)}{2 P \left(\vartheta_u \right)}\right]
\end{align*}
and
\begin{align*}
\Ex_{\vartheta _u}\dot \Pi \left(u\right)^2&=\varphi
_T^2\sum_{t=1}^{T}\Ex_{\vartheta _u}\left[ \frac{\zeta _t\left(\vartheta
    _u\right)f\dot m_{t-1}\left(\vartheta_u \right) }{{ \sqrt{ P
        \left(\vartheta_u \right)} }} +\frac{\left[\zeta
      _t\left(\vartheta _u\right)^2-1\right]\dot P \left(\vartheta_u
    \right)}{2 P \left(\vartheta_u
    \right)}\right]^2\\
 &=\varphi_T^2\sum_{t=1}^{T}\left[
  \frac{f^2\Ex_{\vartheta _u}\left[\zeta _t\left(\vartheta
      _u\right)^2\dot m_{t-1}\left(\vartheta_u \right)^2\right] }{{  P
      \left(\vartheta_u \right) }} +\frac{\dot P \left(\vartheta_u
    \right)^2\Ex_{\vartheta _u}\left[\zeta _t\left(\vartheta
      _u\right)^2-1\right]^2}{4 P \left(\vartheta_u \right)^2}\right]\\
 &=\frac{1}{T}\sum_{t=1}^{T}\left[ \frac{f^2\Ex_{\vartheta _u}\left[\dot
      m_{t-1}\left(\vartheta_u \right)^2\right] }{{  P \left(\vartheta_u
      \right) }} +\frac{\dot P \left(\vartheta_u \right)^2}{2 P
    \left(\vartheta_u \right)^2}\right]\\
 &= \frac{f^2\dot B\left(\vartheta _u,\vartheta _u\right)^2}{{  P \left(\vartheta_u
      \right) }\left(1-A\left(\vartheta _u\right)\right)^2 } +\frac{\dot P \left(\vartheta_u \right)^2}{2 P
    \left(\vartheta_u \right)^2}\leq C.
\end{align*}

Hence we can write
\begin{align*}
&\Ex_{\vartheta _0}\left|Z_T\left(u_2 \right)^{1/2}-Z_T\left(u_1 \right)^{1/2}
  \right|^2=\Ex_{\vartheta _0}\left|\int_{u_1}^{u_2} \frac{\partial }{\partial
    u} Z_T\left(u \right)^{1/2}{\rm d}u \right|^2\\
&\qquad \leq \frac{\left(u_2-u_1\right)}{4}\int_{u_1}^{u_2}\Ex_{\vartheta _0} Z_T\left(u
  \right)\dot \Pi \left(u\right)^2{\rm d}u=
  \frac{\left(u_2-u_1\right)}{4}\int_{u_1}^{u_2}\Ex_{\vartheta _u} 
\dot \Pi \left(u\right)^2{\rm d}u \\
&\qquad \leq C\left(u_2-u_1\right)^2,
\end{align*}
where the constant $C>0$ can be chosen not  depending on $\vartheta _0$.  The estimate \eqref{47} is proved too. 

To verify the last estimate \eqref{48} we write
\begin{align*}
\Pb_{\vartheta _0}\left(Z_T\left(u \right)>e^{-\kappa u^2}
\right)&=\Pb_{\vartheta _0}\left(\Pi \left(u\right) > -\kappa u^2  \right)\\
& =\Pb_{\vartheta _0}\Bigl(\Pi \left(u\right)-\Ex_{\vartheta _0}\Pi \left(u\right) > -\kappa u^2 -\Ex_{\vartheta _0}\Pi \left(u\right) \Bigr) .
\end{align*}
Note that
\begin{align*}
-2\Ex_{\vartheta _0}\Pi \left(u\right)&=\sum_{t=1}^{T}
\frac{\Ex_{\vartheta _0}\Bigl[\zeta _t\left(\vartheta _0\right)\sqrt{ P \left(\vartheta_0 \right)}-f\left[m_{t-1}\left(\vartheta_u 
    \right)-m_{t-1}\left(\vartheta_0 
    \right)\right]\Bigr]^2}{{  P \left(\vartheta_u \right) }}\nonumber\\
&\qquad \qquad \qquad +{T}\ln \frac{ P
    \left(\vartheta_u \right)}{ P \left(\vartheta_0 \right)
}-\sum_{t=1}^{T}\Ex_{\vartheta _0}\zeta _t\left(\vartheta _0\right)^2 \nonumber\\
&=
\frac{f^2}{{  P \left(\vartheta_u \right) }} \sum_{t=1}^{T}   \Ex_{\vartheta _0} \left[m_{t-1}\left(\vartheta_u 
    \right)-m_{t-1}\left(\vartheta_0 
    \right)\right]^2\nonumber\\
&\qquad \qquad \qquad +T\left(1+\ln \frac{ P
    \left(\vartheta_0 \right)}{ P \left(\vartheta_u \right)
}-\frac{ P \left(\vartheta_0 \right)}{  P \left(\vartheta_u \right)}\right).
\end{align*}
Consider two regions of $u$. First suppose that $\left|u\varphi
_\zT\right|\leq \nu $ with some small $\nu >0$. Then using expansions at the
vicinity of $\vartheta _0$ we can write
\begin{align*}
-\Ex_{\vartheta _0}\Pi \left(u\right)&=\left(\frac{u^2f^2\dot B^*\left(\vartheta
  _0,\vartheta _0\right)^2}{{2  P \left(\vartheta_u \right)
    \left(1-A\left(\vartheta _0\right)^2\right)}}+\frac{u^2\dot  P
  \left(\vartheta _0\right)^2}{2 P \left(\vartheta
  _0\right)^2}\right)\left(1+o\left(\nu \right) \right)\nonumber\\
&= \frac{u^2}{2}{\rm
  I}_b\left(\vartheta _0\right)\left(1+o\left(\nu \right) \right)\geq
\frac{u^2}{4}{\rm I}_b\left(\vartheta _0\right)\geq \kappa _1u^2
\end{align*}
for sufficiently small $\nu $. Remark that 
\begin{align*}
\inf_{\vartheta \in\Theta }{\rm I}_b\left(\vartheta\right)>0
\end{align*}
and the constant $\kappa _1$ can be chosen not depending of $\vartheta _0$.

 Let $\left|u\varphi _T\right|\geq \nu
$. Consider the difference of two equations
\begin{align*}
m_{t-1}\left(\vartheta _0\right)&=am_{t-2}\left(\vartheta
_0\right)\,+B^*\left(\vartheta _0,\vartheta _0\right)\zeta _{t-1}\left(\vartheta _0\right),\\
m_{t-1}\left(\vartheta_u \right)&=A\left(\vartheta_u \right)m_{t-2}\left(\vartheta_u
\right)+ \left[a-A\left(\vartheta_u \right)\right]m_{t-2}\left(\vartheta_0
\right)      +B^*\left(\vartheta_u,\vartheta _0\right)\zeta _{t-1}\left(\vartheta _0\right),
\end{align*}  
 and write
\begin{align*}
m_{t-1}\left(\vartheta_u \right)-m_{t-1}\left(\vartheta
_0\right)&=A\left(\vartheta_u \right)\left[m_{t-2}\left(\vartheta_u
  \right)-m_{t-2}\left(\vartheta_0 \right) \right]\\
&\qquad \qquad +\left[B^*\left(\vartheta_u,\vartheta
_0\right)-B^*\left(\vartheta _0,\vartheta _0\right)\right]\zeta
_{t-1}\left(\vartheta _0\right) . 
\end{align*}
Therefore
\begin{align*}
\Ex_{\vartheta _0}\left[m_{t-1}\left(\vartheta_u \right)-m_{t-1}\left(\vartheta
_0\right)\right]^2&=A\left(\vartheta_u \right)^2\Ex_{\vartheta _0}\left[m_{t-2}\left(\vartheta_u
  \right)-m_{t-2}\left(\vartheta_0 \right) \right]^2\\
&\qquad \qquad +\left[B^*\left(\vartheta_u,\vartheta
_0\right)-B^*\left(\vartheta _0,\vartheta _0\right)\right]^2.
\end{align*}
We suppose that $m_{t}\left(\vartheta_u \right) $ and
$m_{t}\left(\vartheta_0 \right) $ are stationary processes. Hence
\begin{align*}
\Ex_{\vartheta _0}\left[m_{t-1}\left(\vartheta_u \right)-m_{t-1}\left(\vartheta
_0\right)\right]^2=\frac{\left[B^*\left(\vartheta_u,\vartheta
_0\right)-B^*\left(\vartheta _0,\vartheta _0\right)\right]^2}{1-A\left(\vartheta_u
  \right)^2} 
\end{align*}
and
\begin{align*}
-2\Ex_{\vartheta _0}\Pi
\left(u\right)&=\frac{Tf^2\left[B^*\left(\vartheta_u,\vartheta
    _0\right)-B^*\left(\vartheta _0,\vartheta _0\right)\right]^2
}{1-A\left(\vartheta_u \right)^2}+T\left(1+\ln \frac{ P \left(\vartheta
  _0\right)}{ P \left(\vartheta _u\right)}-\frac{ P \left(\vartheta
  _0\right)}{ P \left(\vartheta _u\right)} \right).
\end{align*}
Denote
\begin{align*}
G\left(\vartheta _u,\vartheta
_0\right)=\frac{f^2\left[B^*\left(\vartheta_u,\vartheta
    _0\right)-B^*\left(\vartheta _0,\vartheta _0\right)\right]^2
}{2\left(1-A\left(\vartheta_u \right)^2\right)}+\frac{1}{2}\left(1+\ln
\frac{ P \left(\vartheta _0\right)}{ P \left(\vartheta
  _u\right)}-\frac{ P \left(\vartheta _0\right)}{ P \left(\vartheta
  _u\right)}\right)
\end{align*}
and show that
\begin{align*}
g\left(\vartheta _0,\nu \right)=\inf_{\left|\vartheta -\vartheta _0\right|\geq
  \nu }G\left(\vartheta,\vartheta _0\right)>0. 
\end{align*}
To do this it is sufficient to verify that 
\begin{align*}
\inf_{\left|\vartheta -\vartheta _0\right|\geq \nu }\left| P
\left(\vartheta \right)- P \left(\vartheta_0 \right)\right|>0. 
\end{align*}
As $\left| P \left(\vartheta \right)- P \left(\vartheta_0
\right)\right|=f^2\left|\gamma _*\left(\vartheta \right)-\gamma _*\left(\vartheta_0
\right)\right| $ we check the condition $\dot\gamma _*\left(\vartheta
\right)>0$. We have (see \eqref{6})
\begin{align}
\label{52}
\dot\gamma _*\left(\vartheta \right)&=\vartheta 
+ \frac{\left(4\left[f^2\vartheta ^2-\sigma
  ^2\left(1-a^2\right) \right]f^2\vartheta +8\vartheta\sigma
  ^2f^2\right)}{ f^2  \left[\left({\sigma
  ^2\left(1-a^2\right)}{}-f^2\vartheta ^2\right)^2+ 4\vartheta^2\sigma
  ^2f^2\right]^{1/2}     }\nonumber\\
&=\vartheta 
+ \frac{\left[4f^2\vartheta ^2+4\sigma
  ^2\left(1+a^2\right) \right]f^2\vartheta}{ f^2  \left[\left({\sigma
  ^2\left(1-a^2\right)}-f^2\vartheta ^2\right)^2+ 4\vartheta^2\sigma
  ^2f^2\right]^{1/2}     }>0
\end{align}
and
\begin{align*}
\inf_{\vartheta_ 0\in\Theta }g\left(\vartheta _0,\nu \right)>0.
\end{align*}
Therefore, using the relation $T\geq u^2/\left(\beta _b-\alpha _b\right)^2$ we
can write 
\begin{align*}
\inf_{\left|u\right|\geq \nu \sqrt{T}}\Bigl(-\Ex_{\vartheta _0}\Pi \left(u\right)\Bigr)\geq
g\left(\vartheta _0,\nu \right)T \geq \frac{g\left(\vartheta _0,\nu
  \right)u^2}{\left(\beta _b-\alpha _b\right)^2} =\kappa _2u^2.
\end{align*}
We have
\begin{align*}
&\Pi \left(u\right)- \Ex_{\vartheta _0}\Pi
\left(u\right)=\sum_{t=1}^{T} \left[\zeta _t\left(\vartheta
  _0\right)^2-1\right]\left(\frac{ P \left(\vartheta _u\right)- P
    \left(\vartheta _0\right)}{2 P \left(\vartheta
    _u\right)}\right)\\
&\quad\quad \qquad   +\sum_{t=1}^{T}\frac{f \sqrt{ P \left(\vartheta
      _0\right)}}{ P \left(\vartheta _u\right)}{\zeta _t\left(\vartheta _0\right)\left[
      m_{t-1}\left(\vartheta _u\right)-m_{t-1}\left(\vartheta
      _0\right)\right]}\\
&\quad \quad \qquad   -\sum_{t=1}^{T}\frac{f^2\left({ \left[
      m_{t-1}\left(\vartheta _u\right)-m_{t-1}\left(\vartheta
      _0\right)\right]^2- \Ex_{\vartheta _0}\left[
      m_{t-1}\left(\vartheta _u\right)-m_{t-1}\left(\vartheta
      _0\right)\right]^2}{}\right)}{2 P \left(\vartheta _u\right)} \\
& \qquad=\Pi_1 \left(u\right)+\Pi_2 \left(u\right)-\Pi_3 \left(u\right),
\end{align*}
where for any integer $N\geq 1$
\begin{align}
\label{53}
\Ex_{\vartheta _0}\left|\Pi_1 \left(u\right)\right|^{2N}&=\Ex_{\vartheta
  _0}\left|\sum_{t=1}^{T} \left[\zeta _t\left(\vartheta
  _0\right)^2-1\right]\left(\frac{ P \left(\vartheta _u\right)- P
  \left(\vartheta _0\right)}{2 P \left(\vartheta _u\right)}\right)
\right|^{2N}\nonumber\\
 &=\left|u\varphi _T\dot P (\tilde\vartheta _u)
\right|^{2N}\Ex_{\vartheta _0}\left|\sum_{t=1}^{T} \left[\zeta
  _t\left(\vartheta _0\right)^2-1\right] \right|^{2N}\leq
C\left|u\right|^{2N},\\
 \Ex_{\vartheta _0}\left|\Pi_2
\left(u\right)\right|^{2N}&= \frac{f^{2N} { P \left(\vartheta
    _0\right)^{N}}}{ P \left(\vartheta _u\right)^{2N}}\Ex_{\vartheta
  _0}\left|\sum_{t=1}^{T}{\zeta _t\left(\vartheta _0\right)\left[
    m_{t-1}\left(\vartheta _u\right)-m_{t-1}\left(\vartheta _0\right)\right]}
\right|^{2N}\nonumber\\
 &\leq C \left|u\right|^{2N}T^{-N}\Ex_{\vartheta _0}
\left|\sum_{t=1}^{T}{\zeta _t\left(\vartheta _0\right)\dot
  m_{t-1}(\tilde\vartheta _u)} \right|^{2N}\leq C\left|u\right|^{2N},\nonumber
\end{align}
and
\begin{align}
\label{55}
\Ex_{\vartheta _0}\left|\Pi_3
\left(u\right)\right|^{2N}&=\frac{f^{4N}\left|u\right|^{4N}\varphi
  _T^{4N}}{2^{2N}  P \left(\vartheta _u\right)^{2N}} \Ex_{\vartheta
  _0}\left|\sum_{t=1}^{T}\Bigl[\dot m_{t-1}(\tilde\vartheta _u)^2-
  \Ex_{\vartheta _0}\dot m_{t-1}(\tilde\vartheta _u)^2\Bigr]
\right|^{2N}\nonumber \\
&\leq C\left|u\right|^{4N}\varphi  _T^{4N}T^N=C\left|u\right|^{2N}\;\frac{\left|u\right|^{2N}}{T}\leq C\left|u\right|^{2N}.
\end{align}
Here we twice used the estimate
\begin{align*}
\Ex_{\vartheta _0}\left|\sum_{t=1}^{T}K_t\left(\vartheta \right) \right|^{2N}\leq CT^N,
\end{align*}
which is valid for centered Gaussian time series $K_t\left(\vartheta
\right),t\geq 1$ with bounded variance and exponential mixing.  We used as
well the relation $\left|u\varphi _T\right|<\beta _b-\alpha _b$. Remark that
all constants $C>0$ in \eqref{53}-\eqref{55} can be chosen not depending
of $\vartheta _0$.

Denote $\kappa =2\left(\kappa _1\wedge \kappa _2\right)$. Now by Chebyshev
inequality and \eqref{53}-\eqref{55}
\begin{align*}
&\Pb_{\vartheta _0}\left(Z_T\left(u\right)>e^{-\kappa u^2}\right)\leq
\Pb_{\vartheta _0}\left(\left|\Pi \left(u\right)- \Ex_{\vartheta _0}\Pi
\left(u\right)\right|\geq \frac{\kappa u^2}{2}\right)\\
&\qquad \leq \Pb_{\vartheta _0}\left(\left|\Pi_1 \left(u\right)\right|\geq
\frac{\kappa u^2}{6}\right) +\Pb_{\vartheta _0}\left(\left|\Pi_2 \left(u\right)\right|\geq
\frac{\kappa u^2}{6}\right) +\Pb_{\vartheta _0}\left(\left|\Pi_3 \left(u\right)\right|\geq
\frac{\kappa u^2}{6}\right) \\
&\qquad \leq \frac{6^{2N}}{\kappa u^{2N}}\Ex_{\vartheta _0}\left|\Pi_1
\left(u\right) \right|^{2N}+ \frac{6^{2N}}{\kappa u^{2N}}\Ex_{\vartheta _0}\left|\Pi_2
\left(u\right) \right|^{2N}\frac{6^{2N}}{\kappa u^{2N}}\Ex_{\vartheta _0}\left|\Pi_3
\left(u\right) \right|^{2N} \leq \frac{C}{u^{2N}}.
\end{align*}
Therefore the conditions \eqref{45}-\eqref{48} are verified and the
estimators $\hat\vartheta _\zT$ and $\tilde\vartheta _\zT$  by Theorems 3.1.1
and 3.2.1 in \cite{IH81}  have
all mentioned in Theorem \ref{T2} properties.

\end{proof}

\section{Adaptive filter}
We are given  the partially observed system
\begin{align*}
X_t&=f\,Y_{t-1}+\sigma\, w_t,\qquad X_0,\qquad t=1,2,\ldots,\\
Y_t&=a\,Y_{t-1}+ b\, v_t,\qquad \;\;Y_0.
\end{align*}
Recall that if all parameters of the model $\vartheta =\left(f,\sigma
^2,a,b\right)$ are known, then the stationary version $m_t\left(\vartheta \right) $  of the conditional expectation
$m\left(\vartheta ,t\right)=\Ex_\vartheta \left(Y_t|X_s,s\leq t\right)$
satisfies the equation (see \eqref{7}, \eqref{6})
\begin{align}
\label{56}
m_t\left(\vartheta \right)&=am_{t-1}\left(\vartheta
\right)+\frac{af\gamma_*\left(\vartheta \right)}{\sigma
  ^2+f^2\gamma_*\left(\vartheta \right)}\left[X_t-fm_{t-1}\left(\vartheta
  \right)\right],\quad m_0\left(\vartheta \right),\quad  t\geq 1,\\ 
\gamma_*\left(\vartheta
\right)&=\frac{1}{2f^2}\left[{f^2b^2-\sigma ^2\left(1-a^2\right)}\right]
+\frac{1}{2f^2}\left[\left({\sigma
    ^2\left(1-a^2\right)}-b^2f^2\right)^2+4f^2b^2\sigma ^2\right]^{1/2}.\nonumber
\end{align}
Consider the problem of approximation of the random function $m_t\left(\vartheta
\right),t\geq 1 $ when  one of the parameters is unknown.

\subsection{Unknown parameter $b$}

 Suppose that the values $f>0$, $a^2\in [0,1)$ and $\sigma ^2>0$ are known and
   the only unknown parameter is $\vartheta =b\in\Theta =\left(\alpha
   _b,\beta_b\right)$, $\alpha _b>0$.  The proposed below construction is a
   direct analogue of the similar solutions discussed in \cite{Kut22} and
   \cite{Kut23a} for the models in continuous time like
   \eqref{1-8}-\eqref{1-9}.  First  by observations
   $X_0,X_1,\ldots, X_{\tau_ \zT}$ we calculate the MME $\vartheta _{\tau_ \zT}^*=b_{\tau_
     \zT}^* $ , ${\tau_ \zT}=\left[T^\delta \right], \delta \in
   (\frac{1}{2},1)$ (see \eqref{22})
\begin{align}
\label{57}
\vartheta _{\tau_ \zT}^*=2^{-1/2}f^{-1}\left[ \left(\frac{1}{\tau _\zT}\sum_{s=1}^{\tau _\zT} \left[X_s-X_{s-1}\right]^2 +\sigma ^2\right)\left(1+a\right)\right]^{1/2} .
\end{align}
Then we define the One-step MLE-process (see \eqref{36})
\begin{align}
\label{58}
\vartheta _{t,T}^\star&=\vartheta _{\tau _\zT}^*+\frac{1}{{\rm
      I}_b(\vartheta _{\tau _\zT}^*)\left(t-\tau _\zT\right) }\sum_{s=\tau
    _\zT+1}^{t}\left[\frac{\left[X_s-fm_{s-1}(\vartheta_{\tau _\zT}^*
        )\right]}{ P (\vartheta_{\tau _\zT}^*)}f\dot
    m_{s-1}(\vartheta_{\tau _\zT}^* )\right.\nonumber\\
 &\quad  \left.+   \left( \left[X_s-fm_{s-1}(\vartheta_{\tau _\zT}^*
     )\right]^2- P(\vartheta_{\tau _\zT}^* )\right)\frac{ \dot
       P(\vartheta_{\tau _\zT}^* )
    }{2 P(\vartheta_{\tau _\zT}^* )^2}\right],\quad t\in \left[ \tau
    _\zT+1,T\right].
\end{align}
Here $m_{s-1}(\vartheta_{\tau _\zT}^*,{s-1} ),s\geq \tau _\zT+1 $ satisfies the
equation \eqref{56}, where $b$ is replaced by $\vartheta _{\tau _\zT}^* $ and
$\dot m_{s-1}(\vartheta_{\tau _\zT}^*,{s-1} ),s\geq \tau _\zT+1 $ is obtained by
differentiation of \eqref{56} on $b$ and once more replacing $b$ by
$\vartheta _{\tau _\zT}^* $. Therefore for $s>\tau _\zT+1$ we can write 
\begin{align}
\label{59}
m_{s-1}(\vartheta_{\tau _\zT}^* )&=P(\vartheta_{\tau
    _\zT}^* )^{-1}   \left[ a\sigma ^2m_{s-2}(\vartheta_{\tau _\zT}^*
 )+af\gamma_*(\vartheta_{\tau _\zT}^*) X_{s-1} \right]  ,\\ 
\dot m_{s-1}(\vartheta_{\tau _\zT}^*)&=P(\vartheta_{\tau _\zT}^* )^{-2}
     {a\sigma ^2 \left[P(\vartheta_{\tau _\zT}^* ) \dot m_{s-2}(\vartheta_{\tau _\zT}^* ,s-2)-f^2\dot \gamma _* (\vartheta_{\tau _\zT}^* )\, m_{s-2}(\vartheta_{\tau _\zT}^* )   \right] }
\nonumber\\ & \qquad +\frac{af\sigma ^2\dot\gamma _*\left(\vartheta
  \right)}{P\left(\vartheta \right)^{2}}\, X_{s-1} .
\end{align}

 Here $P\left(\vartheta
\right)=\sigma ^2+f^2\gamma _*\left(\vartheta \right)$ and  the Fisher information
is 
\begin{align*}
{\rm I}_b(\vartheta )=\frac{\dot P_b\left(\vartheta_0
\right)^2 \left[P\left(\vartheta_0
\right)^2   +a^2\sigma ^4 \right]  }{2 P
  \left(\vartheta_0\right)^2 \left[P\left(\vartheta_0
\right)^2  -a^2\sigma ^4 \right]},
\end{align*}
where $A\left(\vartheta \right)=P\left(\vartheta \right)^{-1}a\sigma ^2 $ (see \eqref{28}).

The adaptive Kalman filter we introduce with the help of the process $
m^\star_{t,T}$ defined as follows
\begin{align}
\label{61}
\hat m_{t,T}&=P(\vartheta _{t-1,T}^\star)^{-1}\left[  a\sigma ^2m^\star_{t-1,T}+af\gamma_*(\vartheta_{t-1,T}^\star)   X_{t} \right],\qquad t\in
\left[\tau _\zT+1,T\right] .
\end{align}
We have to compare $m^\star_{t,T} $ with $m\left(\vartheta _0,t\right)$ for
large values of $T$.
Below $\eta _{t,T}= \sqrt{t}\left[\vartheta _{t,T}^\star -\vartheta _0
  \right]$ and      $\dot B^*\left(\vartheta _0,\vartheta _0\right)=-f^{-1}\sqrt{P\left(\vartheta
  _0\right)}\dot A(\vartheta _0)=P\left(\vartheta _0\right)^{-3/2}af\sigma
^2\dot \gamma _*\left(\vartheta _0\right) $.

\begin{theorem}
\label{T3} Let $t=\left[vT\right], v\in
(0,1]$, $k=t-\tau _\zT+2$  and $T\rightarrow \infty $. Then the following relations  hold
\begin{align}
&\sqrt{t}\left[m^\star_{t,T}-m_ {t} \left(\vartheta _0\right)\right]=\dot
B^*\left(\vartheta _0,\vartheta _0\right)\sum_{m=0}^{k}A\left(\vartheta
_0\right)^m\eta _{t-m-1,T}\;\zeta _{t-m}\left(\vartheta
_0\right)+o\left(1\right),\nonumber\\ &t\Ex_{\vartheta
  _0}\left[m^\star_{t,T}-m_{t}\left(\vartheta
  _0\right)\right]^2\longrightarrow S_b^\star\left(\vartheta _0\right)^2\equiv \frac{\dot B^*\left(\vartheta _0,\vartheta
  _0\right)^2}{{\rm I}_b\left(\vartheta _0\right)\left(1-A\left(\vartheta
  _0\right)^{2}\right)}.
\label{63}
\end{align}
Here $o\left(1\right)$ and  \eqref{63} converge  uniformly on compacts
$\KK\subset\Theta $. 
\end{theorem}
\begin{proof}
Recall that
\begin{align*}
m_{t}\left(\vartheta _0\right)&=A(\vartheta _0) m_{t-1}\left(\vartheta _0\right)+
f^{-1}\left[a-A(\vartheta _0) \right] X_{t} ,\\
m^\star_{t,T}&=A(\vartheta_{t-1,T}^\star) \hat m_{t-1,T}+
f^{-1}\left[a-A(\vartheta _{t-1,T}^\star) \right] X_{t}.
\end{align*}
Therefore for $\delta _{t,T}=\sqrt{t}\left[m^\star_{t,T}-m_{t}\left(\vartheta _0\right)\right]$ we have
the equation 
\begin{align*}
\delta _{t,T}&=A(\vartheta_{t-1,T}^\star) \delta _{t-1,T}+\sqrt{t}
\left[A(\vartheta_{t-1,T}^\star)-A(\vartheta _0)\right]m_{t-1}\left(\vartheta
_0\right) \\
&\qquad \qquad  +f^{-1}\sqrt{t}\left[A(\vartheta _0)-A(\vartheta _{t-1,T}^\star) \right] X_{t}\\
&= A(\vartheta_{t-1,T}^\star) \delta _{t-1,T} +f^{-1}\sqrt{P\left(\vartheta
  _0\right)}\sqrt{T}\left[A(\vartheta _0)-A(\vartheta _{t-1,T}^\star) \right]\zeta
_t\left(\vartheta _0\right)\\ 
&= A(\vartheta_{t-1,T}^\star) \delta _{t-1,T}
-f^{-1}\sqrt{P\left(\vartheta _0\right)}\dot A(\tilde\vartheta
_{t-1})\sqrt{t}\left[\vartheta _{t-1,T}^\star -\vartheta _0 \right]\zeta
_t\left(\vartheta _0\right)\\ 
&= A(\vartheta_0)\; \delta _{t-1,T} +\dot B^*\left(\vartheta _0,\vartheta _0\right)\;\eta _{t-1,T}\;\zeta _t\left(\vartheta
_0\right)+\varepsilon _t,\qquad t\in \left[\tau _\zT+1,T\right],
\end{align*}
where  $\varepsilon _t=O\left(\vartheta _{t-1,T}^\star -\vartheta _0
\right)=O\left(T^{-1/2}\right)$. 

 Let us denote $A\equiv A\left(\vartheta
_0\right), \dot B^*\equiv \dot B^*\left(\vartheta _0,\vartheta _0\right), \zeta _t=\zeta _t\left(\vartheta
_0\right)$ and make $k=t-\tau _\zT+2$ iterations
\begin{align*}
\delta _{t,T}&=A\, \delta _{t-1,T}+\dot B^*\,\eta _{t-1,T}\;\zeta _t+\varepsilon _t\\
&=A^2\, \delta _{t-2,T}+  A \dot B^*\,\eta _{t-2,T}\;\zeta _{t-1}+\dot B^*\,\eta
_{t-1,T}\;\zeta _t+A\varepsilon _{t-1}+ \varepsilon _t \\
&=\ldots .\ldots .\ldots .\ldots .\ldots .\ldots .\ldots .\ldots .\ldots .\ldots .\ldots .\ldots .\\
&=A^{k+1}\delta _{t-k,T}+\dot B^*\sum_{m=0}^{k}A^m\eta _{t-m-1,T}\;\zeta _{t-m}+\sum_{m=0}^{k}A^m\varepsilon _{t-m}.
\end{align*}  
Therefore for the large values of $t$ and $k=t-\tau _\zT+2$ we have the   representation, 
\begin{align*}
\delta _{t,T}&=\dot B^*\sum_{m=0}^{k}A^m\eta _{t-m-1,T}\;\zeta
_{t-m}+o\left(1\right).
\end{align*}
Note that $\eta _{t-m-1,T} $ and $\zeta
_{t-m}$ are independent. Hence, if $t=\left[vT\right], T\rightarrow \infty $, then
\begin{align*}
\Ex_{\vartheta _0} \delta _{t,T}^2&=\dot B^{*2}\sum_{m=0}^{k}\sum_{l=0}^{k}A^{m+l}\Ex_{\vartheta _0} \left[\eta _{t-m-1,T}\;\zeta
_{t-m} \eta _{t-l-1,T}\;\zeta
_{t-l}\right]+o\left(1\right)\\
&=\dot B^{*2}\sum_{m=0}^{k}A^{2m }\Ex_{\vartheta _0} \left[\eta _{t-m-1,T}^2\zeta
_{t-m}^2\right]+o\left(1\right)=\dot B^{*2}\sum_{m=0}^{k}A^{2m }\Ex_{\vartheta _0} \left[\eta _{t-m-1,T}^2\right]+o\left(1\right)\\
&=\frac{\dot B^{*2}}{{\rm I}_b\left(\vartheta _0\right)}\sum_{m=0}^{k}A^{2m
}\frac{t}{t-m-1}+o\left(1\right)\\
&\longrightarrow \frac{\dot B^*\left(\vartheta
  _0,\vartheta
  _0\right)^{2}}{{\rm I}\left(\vartheta
  _0\right)}\sum_{m=0}^{\infty }A^{2m}= \frac{\dot B^*\left(\vartheta
  _0,\vartheta
  _0\right)^2}{{\rm I}_b\left(\vartheta _0\right)\left(1-A\left(\vartheta
  _0\right)^{2}\right)}. 
\end{align*}

\end{proof}
\begin{remark}
\label{R8}
{\rm The adaptive Kalman filter is given by the relations
  \eqref{57}-\eqref{61}, where the only One-step MLE-process
  \eqref{58} is in non recurrent form. Let us write this estimator in
  recurrent form too.  Denote
\begin{align*}
S_{t,T}(\vartheta _{\tau _\zT}^*)&=\frac{1}{{\rm
      I}_b(\vartheta _{\tau _\zT}^*)\left(t-\tau _\zT\right) }\sum_{s=\tau
    _\zT+1}^{t}\left[\frac{\left[X_s-fm_{s-1}(\vartheta_{\tau _\zT}^*
        )\right]}{ P (\vartheta_{\tau _\zT}^*)}f\dot
    m_{s-1}(\vartheta_{\tau _\zT}^* )\right.\nonumber\\
 &\qquad \qquad \quad  \left.+   \left( \left[X_s-fm_{s-1}(\vartheta_{\tau _\zT}^*
     )\right]^2- P(\vartheta_{\tau _\zT}^* )\right)\frac{ \dot
       P(\vartheta_{\tau _\zT}^* )
    }{2 P(\vartheta_{\tau _\zT}^* )^2}\right].
\end{align*}
Then we can write 
\begin{align*}
 \vartheta _{t,T}^\star&=\vartheta _{\tau _\zT}^*+ S_{t,T}(\vartheta _{\tau
   _\zT}^*)=\frac{\vartheta _{\tau _\zT}^*}{t-\tau _\zT}+ \frac{\left(t-1-\tau
   _\zT\right) }{\left(t-\tau _\zT\right) } \left[\vartheta _{\tau _\zT}^*+ S_{t-1,T}(\vartheta _{\tau
   _\zT}^*)\right]\nonumber\\
&\qquad \quad +\frac{1}{{\rm
      I}_b(\vartheta _{\tau _\zT}^*)\left(t-\tau _\zT\right) }\left[\frac{\left[X_t-fm_{t-1}(\vartheta_{\tau _\zT}^*
        )\right]}{ P (\vartheta_{\tau _\zT}^*)}f\dot
    m_{t-1}(\vartheta_{\tau _\zT}^* )\right.\nonumber\\
 &\qquad\quad  \left.+   \left( \left[X_t-fm_{t-1}(\vartheta_{\tau _\zT}^*
    )\right]^2- P(\vartheta_{\tau _\zT}^* )\right)\frac{ \dot
       P(\vartheta_{\tau _\zT}^* )
    }{2 P(\vartheta_{\tau _\zT}^* )^2}\right]\nonumber\\
&=\frac{\vartheta _{\tau _\zT}^*}{t-\tau _\zT}+ \left(1-\frac{1}{t-\tau
   _\zT}\right)  \;  \vartheta _{t-1,T}^\star  +\frac{ \left[X_t-fm_{t-1}(\vartheta_{\tau _\zT}^*
        )\right]f\dot
    m_{t-1}(\vartheta_{\tau _\zT}^* ) }{{\rm
      I}_b(\vartheta _{\tau _\zT}^*)\left(t-\tau _\zT\right)P(\vartheta_{\tau _\zT}^* )}\nonumber\\
 &\qquad\quad  +\frac{    \left( \left[X_t-fm_{t-1}(\vartheta_{\tau _\zT}^*
    )\right]^2- P(\vartheta_{\tau _\zT}^* )\right)\dot
       P(\vartheta_{\tau _\zT}^* )  }{2{\rm
      I}_b(\vartheta _{\tau _\zT}^*)\left(t-\tau _\zT\right) P(\vartheta_{\tau
     _\zT}^* )^2}, \quad t\in \left[ \tau _\zT+1,T\right].
\end{align*}
Therefore the recurrent equation for the One-step MLE-process is 
\begin{align}
\label{64}
\vartheta _{t,T}^\star&=\frac{\vartheta _{\tau _\zT}^*}{t-\tau _\zT}+ \left(1-\frac{1}{t-\tau
   _\zT}\right)  \;  \vartheta _{t-1,T}^\star  +\frac{ \left[X_t-fm_{t-1}(\vartheta_{\tau _\zT}^*
        )\right]f\dot
    m_{t-1}(\vartheta_{\tau _\zT}^* ) }{{\rm
      I}_b(\vartheta _{\tau _\zT}^*)\left(t-\tau _\zT\right)P(\vartheta_{\tau _\zT}^* )}\nonumber\\
 &\qquad +\frac{    \left( \left[X_t-fm_{t-1}(\vartheta_{\tau _\zT}^*
    )\right]^2- P(\vartheta_{\tau _\zT}^* )\right)\dot
       P(\vartheta_{\tau _\zT}^* )  }{2{\rm
      I}_b(\vartheta _{\tau _\zT}^*)\left(t-\tau _\zT\right) P(\vartheta_{\tau
     _\zT}^* )^2}, \quad t\in \left[ \tau _\zT+1,T\right].
\end{align}

Now the adaptive Kalman filter \eqref{57},\eqref{59} -\eqref{61},
\eqref{64} is in recurrent form.

}
\end{remark}

\begin{remark}
\label{R9}
{\rm The cases $\vartheta =f$, $\vartheta =a$ and $\vartheta =\sigma ^2$ can
  be studied similarly. If we consider the two-dimensional cases, say,
  $\vartheta =\left(f,a\right)$,  then the corresponding adaptive Kalman filter
  can be written as well. The preliminary MME estimator $\vartheta _{\tau
    _\zT}^* =\left(f_T^*,a_T^*\right)^\top$ was defined in
  \eqref{24}-\eqref{25}, for the Fisher information matrix see
  \eqref{32},\eqref{33} and \eqref{35}, One-step MLE-process is given
  in \eqref{41}, the equations for  $\dot m_f\left(\vartheta _\zT,t\right)
  $ and  $\dot m_a\left(\vartheta _\zT,t\right)  $ can be easily written. The
  equation for  $m^\star_{t,T}$  has exactly the same form as the given above
  in \eqref{61}.
 
}
\end{remark}

\begin{remark}
\label{R10}
{\rm  
It is possible as well to verify the asymptotic normality 
\begin{align*}
\sqrt{t}\left(m^\star_{t,T}-m\left(\vartheta _0,t\right)\right)\Longrightarrow
     {\cal N}\left(0,S_b^\star\left(\vartheta _0\right)^2\right).
\end{align*}

}
\end{remark}
\section{Asymptotic efficiency}

We have the same  system
\begin{align*}
X_t&=f\,Y_{t-1}+\sigma\, w_t,\qquad X_0,\qquad t=1,2,\ldots,\\
Y_t&=a\,Y_{t-1}+ b\, v_t,\qquad \;\;Y_0.
\end{align*}
where the parameters  $f,a,b,\sigma ^2$ satisfy the condition ${\scr A}_0$. 

The adaptive filter is given by \eqref{61} and we would like to know if the
error of approximation $\Ex_\vartheta \left|m_{t,T}^\star-m\left(\vartheta,t
\right)\right|^2 $ is asymptotically minimal?  As usual in such situations we
propose a lower minimax bound on the risks of all estimators $\bar m_t$
supposing that $\bar m_t$   is based on the observations up to time $t$,
i.e., these estimators depend on $X^t=\left(X_s,s=0,1,\ldots, t\right)$.

Recall some notations 
\begin{align*}
A\left(\vartheta _0\right)=\frac{a\sigma ^2}{\sigma ^2+f^2\gamma
  _*\left(\vartheta _0\right)}, \qquad B^*\left(\vartheta _0,\vartheta
_0\right)= \frac{a\sigma ^2f\dot \gamma _*\left(\vartheta
  _0\right)}{\left[\sigma ^2+ f^2\gamma _*\left(\vartheta _0\right)\right]^{3/2}}
\end{align*}
and the equations for $\dot m\left(\vartheta_0 ,\cdot
\right) $
\begin{align*}
\dot m\left(\vartheta ,t \right)&=A\left(\vartheta _0\right)\dot m\left(\vartheta_0 ,t-1 \right) +B^*\left(\vartheta _0,\vartheta _0\right)\zeta _t\left(\vartheta _0\right),
\end{align*}

Recall as well the asymptotic representations for
$\eta _{t,T}^\star\left(\vartheta _0\right)=\sqrt{vT}\left(\vartheta
_{t,T}^\star-\vartheta _0\right), t=vT, v\in (0,1] $
\begin{align}
\label{42-75}
 \eta _{t,T}^\star\left(\vartheta _0\right)=\frac{{\rm I}\left(\vartheta _0\right)^{-1}}{\sqrt{t P\left(\vartheta_0
    \right)}}\sum_{s=1}^{t} \left[ \zeta _s\left(\vartheta_0
  \right)f\dot m_{s-1}\left(\vartheta_0 \right)+\left[\zeta
    _s\left(\vartheta_0 \right)^2-1 \right]\frac{\dot P \left(\vartheta_0
    \right)}{2\sqrt{ P \left(\vartheta_0 \right)}}\right]+o\left(1\right).
\end{align}
 The similar representations for the MLE $\hat\eta
_{t}\left(\vartheta _0\right)=\sqrt{t}\left(\hat\vartheta
_{t }-\vartheta _0\right), v\in (0,1] $ and BE $\tilde\eta
  _{t}\left(\vartheta _0\right)=\sqrt{t}\left(\tilde\vartheta
  _{t}-\vartheta _0\right),t=vT,  v\in (0,1] $ are 
\begin{align}
 \hat\eta _{t}\left(\vartheta _0\right)&=\frac{{\rm I}\left(\vartheta
   _0\right)^{-1}}{\sqrt{t P\left(\vartheta_0 
    \right)}}\sum_{s=1}^{t} \left[ \zeta _s\left(\vartheta_0
  \right)f\dot m_{s-1}\left(\vartheta_0 \right)+\left[\zeta
    _s\left(\vartheta_0 \right)^2-1 \right]\frac{\dot P \left(\vartheta_0
    \right)}{2\sqrt{ P \left(\vartheta_0 \right)}}\right]+o\left(1\right),\nonumber\\
\label{42-80}
 \tilde\eta _{t}\left(\vartheta _0\right)&=\frac{{\rm I}\left(\vartheta _0\right)^{-1}}{\sqrt{t P\left(\vartheta_0
    \right)}}\sum_{s=1}^{t} \left[ \zeta _s\left(\vartheta_0
  \right)f\dot m_{s-1}\left(\vartheta_0 \right)+\left[\zeta
    _s\left(\vartheta_0 \right)^2-1 \right]\frac{\dot P \left(\vartheta_0
    \right)}{2\sqrt{ P \left(\vartheta_0 \right)}}\right]+o\left(1\right).
\end{align}
The properties \eqref{45}-\eqref{48}  of the normalized LR  $Z_T\left(\cdot
\right)$ established in the Theorem \ref{T2} correspond to the sufficient
conditions of Theorems 3.1.2  and 3.2.2 in \cite{IH81}, where such
representations were proved in the general case.

Introduce the limit \eqref{63}
\begin{align}
\label{42-81}
S^\star_b\left(\vartheta\right)^2=\lim_{T\rightarrow \infty } \Ex_{\vartheta }\left[ \dot
m\left(\vartheta,t \right)^2 \tilde\eta _{t}\left(\vartheta
\right)^2\right]=\frac{\dot B^*\left(\vartheta ,\vartheta \right)^2}{{\rm I}\left(\vartheta \right)\left(1-A\left(\vartheta \right)^2\right) }.
\end{align}
As the asymptotic representations of $\tilde\eta _{t,T}\left(\vartheta \right)
$ and $ \eta _{t,T}^\star\left(\vartheta \right)$ are similar (see
\eqref{42-75} and \eqref{42-80}) hence the limits \eqref{63} and \eqref{42-81}
coincide too.

\begin{theorem}
\label{T42-6}
Let the conditions of Theorem \ref{T2} be fulfilled. Then 
we have the following lower minimax bound: for any estimator $\bar m_{t,T}$ of
$m\left(\vartheta,t \right)$ (below $t=vT$)
\begin{align*}
\lim_{\nu \rightarrow 0}\Liminf_{T\rightarrow \infty }\sup_{\left|\vartheta
  -\vartheta _0\right|\leq \nu }t \Ex_{\vartheta }\left|\bar
m_{t,T}-m\left(\vartheta,t \right) \right|^2 \geq S_b^\star\left(\vartheta _0\right)^2.
\end{align*}

\end{theorem}
\begin{proof} The given below proof is based on the proof of Theorem 1.9.1 in \cite{IH81} and 
was published in \cite{Kut23a} in the case of continuous time observations. It
is quite short and we repeat it here for convenience of reading.  We have the
elementary estimate
\begin{align*}
\sup_{\left|\vartheta -\vartheta _0\right|\leq \nu } \Ex_{\vartheta
}\left|\bar m_{t,T}-m\left(\vartheta,t \right) \right|^2 \geq \int_{\vartheta
  _0-\nu }^{\vartheta _0+\nu }\Ex_{\vartheta }\left|\bar
m_{t,T}-m\left(\vartheta,t \right) \right|^2p_\nu \left(\vartheta \right){\rm
  d}\vartheta.
\end{align*}
Here the function $p_\nu \left(\vartheta \right),\vartheta _0-\nu<\vartheta
<\vartheta _0+\nu $ is a positive continuous density on the interval
$\left[\vartheta _0-\nu,\vartheta _0+\nu \right]$. If we denote
$\tilde m_{t}$ Bayesian estimator of $m\left(\vartheta,t \right) $, which
corresponds to this density $p_\nu \left(\cdot \right)$, then
\begin{align*}
\tilde m_{t}=\int_{\vartheta _0-\nu }^{\vartheta _0+\nu }m\left(\theta,t
\right) p_\nu \left(\theta |X^t\right){\rm d}\theta  ,\qquad p_\nu
\left(\theta |X^t\right)=\frac{p_\nu \left(\theta
  \right)L\left(\theta ,X^t\right) }{\int_{\vartheta _0-\nu }^{\vartheta
    _0+\nu }p_\nu \left(\theta 
  \right)L\left(\theta ,X^t\right){\rm d}\theta   }
\end{align*}
and
\begin{align*}
\int_{\vartheta
  _0-\nu }^{\vartheta _0+\nu }\Ex_{\vartheta }\left|\bar
m_{t,T}-m\left(\vartheta,t \right) \right|^2p_\nu \left(\vartheta \right){\rm
  d}\vartheta\geq \int_{\vartheta
  _0-\nu }^{\vartheta _0+\nu }\Ex_{\vartheta }\left|\tilde
m_{t}-m\left(\vartheta,t \right) \right|^2p_\nu \left(\vartheta \right){\rm
  d}\vartheta.
\end{align*}
The asymptotic behavior of BE $\tilde m_{t} $ can be described as follows (below $\theta
_u=\vartheta+\varphi _t u, \varphi _t=t^{-1/2}, \UU_\nu
=\left(\sqrt{t}\left(\vartheta _0-\nu -\vartheta
\right),\sqrt{t}\left(\vartheta _0+\nu -\vartheta \right) \right)$    )  
\begin{align*}
\tilde m_{t}&=\frac{\int_{\vartheta _0-\nu }^{\vartheta _0+\nu }m\left(
  \theta,t \right){p_\nu \left(\theta \right)L\left(\theta ,X^t\right) }{\rm
    d}\theta } {\int_{\vartheta _0-\nu }^{\vartheta _0+\nu }p_\nu \left(\theta
  \right)L\left(\theta ,X^t\right){\rm d}\theta }=\frac{\int_{\UU_\nu }^{
  }m\left(\theta_u,t \right){p_\nu \left(\theta_u
    \right)L\left(\theta_u,X^t\right) }{\rm d}u } {\int_{\UU_\nu }^{}p_\nu
  \left(\theta_u \right)L\left(\theta_u ,X^t\right){\rm d}u }\\
 &=m\left(\vartheta,t \right)+\varphi _t\dot m\left( \vartheta,t \right)
\frac{\int_{\UU_\nu }^{ }u{p_\nu \left(\theta_u
    \right)\frac{L\left(\theta_u,X^t\right)}{L\left(\vartheta ,X^t\right)}
  }{\rm d}u } {\int_{\UU_\nu }^{}p_\nu \left(\theta_u
  \right)\frac{L\left(\theta_u,X^t\right)}{L\left(\vartheta ,X^t\right)}{\rm
    d}u }\left(1+o\left(1\right)\right)\\
 &=m\left(\vartheta,t \right)+\varphi _t\dot m\left( \vartheta,t \right)
\frac{\int_{\UU_\nu } u  p_\nu \left(\vartheta \right)Z_t\left(u\right){\rm
      d}u } {\int_{\UU_\nu }^{}p_\nu \left(\vartheta
    \right)Z_t\left(u\right){\rm d}u }\left(1+o\left(1\right)\right).
\end{align*}
Hence
\begin{align*}
\sqrt{t}\left(\tilde m_{t}- m\left(\vartheta,t \right)
\right)&=\dot m\left( \vartheta,t \right)\frac{\int_{\UU_\nu } u
  Z_t\left(u\right){\rm d}u } {\int_{\UU_\nu 
  }^{} Z_t\left(u\right){\rm d}u }\left(1+o\left(1\right)\right)\\
&=\dot m\left( \vartheta,t \right)\frac{\Delta _t\left(\vartheta
  ,X^t\right)}{{\rm I}\left(\vartheta \right)} \left(1+o\left(1\right)\right) ,
\end{align*}
where  (see Lemma \ref{L1})
\begin{align*}
\Delta _t\left(\vartheta
  ,X^t\right)=\frac{1}{\sqrt{t P\left(\vartheta_0
    \right)}}\sum_{s=1}^{t} \left[ \zeta _s\left(\vartheta_0
  \right)f\dot m_{s-1}\left(\vartheta_0 \right)+\left[\zeta
    _s\left(\vartheta_0 \right)^2-1 \right]\frac{\dot P \left(\vartheta_0
    \right)}{2\sqrt{ P \left(\vartheta_0 \right)}}\right].
\end{align*}
Recall that $Z_t\left(u\right)\Rightarrow Z\left(u\right)=\exp\left(u\Delta
\left(\vartheta \right)-\ds\frac{u^2}{2}{\rm I}\left(\vartheta \right) \right)
$
and
\begin{align*}
\frac{\int_{{\cal R} } u Z\left(u\right){\rm d}u } {\int_{{\cal R} }^{}
  Z\left(u\right){\rm d}u }=\frac{\Delta
\left(\vartheta \right)}{{\rm I}\left(\vartheta \right) }.
\end{align*}
Moreover  the uniform  on compacts $\KK\subset \left(\vartheta _0-\nu
,\vartheta _0+\nu\right)$ convergence of moments of BE allows us to write 
\begin{align*}
t\Ex_\vartheta \left(\tilde m_{t}- m\left(\vartheta,t \right) \right)^2
\rightarrow \lim_{t\rightarrow \infty }t \Ex_\vartheta \left[\dot m\left(
  \vartheta,t \right)^2(\tilde\vartheta _t-\vartheta
  )^2\right]=S_b^\star\left(\vartheta \right)^2
\end{align*}
holds too. 

The detailed proof of written above relations can be found in the proofs of
Theorems 3.2.1 and 3.2.2 in \cite{IH81} 

Therefore
\begin{align*}
t\int_{\vartheta _0-\nu }^{\vartheta _0+\nu }\Ex_{\vartheta }\left|\tilde
m_{t}-m\left(\vartheta,t \right) \right|^2p_\nu \left(\vartheta \right){\rm
  d}\vartheta \longrightarrow \int_{\vartheta _0-\nu }^{\vartheta _0+\nu
}S_b^\star\left(\vartheta \right)^2p_\nu \left(\vartheta \right){\rm d}\vartheta 
\end{align*}
and as $\nu \rightarrow 0$ 
\begin{align*}
\int_{\vartheta _0-\nu }^{\vartheta _0+\nu
}S_b^\star\left(\vartheta \right)^2p_\nu \left(\vartheta \right){\rm d}\vartheta
\longrightarrow S_b^\star\left(\vartheta_0 \right)^2 .
\end{align*}

\end{proof}

  We call the estimator $ m_{t,T}^\circ, \tau _\zT<t\leq T $   {\it
  asymptotically efficient} if for all $\vartheta _0\in\Theta $, $t=vT$, any  $v\in
  \left[\varepsilon _0,1\right]$
\begin{align*}
\lim_{\nu \rightarrow 0}\lim_{T\rightarrow \infty }\sup_{\left|\vartheta
  -\vartheta _0\right|\leq \nu }t \Ex_{\vartheta }\left|
m_{t,T}^\circ -m\left(\vartheta,t \right) \right|^2 = S_b^\star\left(\vartheta _0\right)^2.
\end{align*}

Here $\varepsilon _0\in (0,1)$.

\begin{theorem}
\label{T5}
The  estimator  $m^\star_t,  \tau _\zT<t\leq T$
is    asymptotically efficient.
\end{theorem}
\begin{proof}
The proof follows from the uniform convergence  \eqref{63}  of the Theorem \ref{T3}.

\end{proof}



\begin{thebibliography}{99}

\bibitem{AWD10} Almagbile, A., Wang, J. and Ding, W. (2010) Evaluating the
  performances of adaptive Kalman filter methods in GPS/INS integration. {\it
    J. Global Positioning Systems.} 9,  1,  33-40.


\bibitem{BC93}  Bell, B.M. and Cathey, F. (1993)   The iterated Kalman filter update as a
Gauss-Newton method. {\it  IEEE Trans. Autom. Control},  38,  2, 294-297.

\bibitem {BRR98} Bickel, P.J., Ritov, Y. and Ryd\'en, T. (1998) Asymptotic
  normality of the maximum likelihood estimator for general hidden Markov
  models. {\it Ann. Statist.}, 26, 4, 1614-1635.

\bibitem{BR85} Brown, S.D. and Rutan, S.C. (1985) Adaptive Kalman
  filtering. {\it Journal of  Research of the national bureau of standards.}
  90, 6, 403-407.

\bibitem{CMR05} Capp\'e, O., Moulines, E. and Ryd\'en, T. (2005) {\it Inference
  in Hidden Markov Models}. Springer, N.Y.

\bibitem{CC87} Chui, C.K. and Chen, G. (1987) {\it   Kalman Filtering with
  Real-time Applications.} Springer, Berlin.

\bibitem{DD07} Dedecker, J., Doukhan, P., Lang, G., Leon, J.R., Louhichi,
  S. and Prieur, C. (2007) {\it Weak Dependence: With Examples and
    Applications.} Springer, N.Y.


\bibitem{Hay14} Haykin, S. (2014) {\it Adaptive Filter Theory.} Fifth Ed.,
  Pearson Education, Boston.  


\bibitem{IH81} Ibragimov, I.A. and Khasminskii R. Z. (1981)
 {\it Statistical Estimation --- Asymptotic Theory.} Springer, N.Y.

\bibitem{Jaz70} Jazwinski, A.H. (1970) {\it Stochastic Processes and Filrering
  Theory.} Academic Press, N.Y. 

\bibitem{Kal60} Kalman, R.E. (1960) A new approach to linear filtering and
  prediction problems. {\it ASME Journal of Basic Engineering,} 82, 35-45. 







\bibitem{KhK18} Khasminskii, R.Z. and Kutoyants, Yu.A.  (2018) On parameter
    estimation of hidden telegraph process.  {\it Bernoulli}, 24, 3,
  2064-2090.

 

\bibitem{Kut94} Kutoyants, Yu.A. (1994) {\it Identification of Dynamical
  Systems with Small Noise.} Kluwer Academic Publisher, Dordrecht.



 

\bibitem{Kut19b} Kutoyants, Yu.A.  (2019)  On parameter estimation of the
  hidden   ergodic Ornstein--Uhlenbeck process. {\it Electronic Journal of
    Statistics}, 13, 4508-4526. 

\bibitem{Kut19a} Kutoyants, Yu.A.  (2019)  On parameter estimation of the hidden
  Ornstein--Uhlenbeck process. {\it J. Multivar. Analysis}, 169, 1, 248-269.


\bibitem{Kut22} Kutoyants, Yu. A.  (2022) Volatility estimation of hidden
  Markov process and adaptive filtration", {\it arXiv:2010.07603}, submitted

\bibitem{Kut23a} Kutoyants, Yu. A.  (2023) Hidden ergodic Ornstein-Uhlenbeck
  process and adaptive filter. {\it arXiv:2304.08857}, submitted.  



\bibitem{KZ21} Kutoyants, Yu.A. and Zhou, L. (2021) On parameter estimation
  of the hidden Gaussian process in perturbed SDE, {\it Electr. J. of
    Stat.}, 15, 211-234.

\bibitem{LS01} Liptser, R.S. and  Shiryayev, A.N.  (2001) {\it Statistics of
  Random Processes, I. General Theory.} 2nd Ed., Springer, N.Y.




\bibitem{MS99}  Mohamed, A.H. and  Schwarz,  K.P.   (1999) Adaptive Kalman
  filtering for INS/GPS  {\it J.  Geod.,} 73, May,   193-203.


\bibitem{S08}   Sayed, A.H. (2008) {\it Adaptive Filters.} John Wiley \& Sons,
  New Jersey.

\bibitem{YG06} Yang, Y. and Gao, W. (2006) An optimal adaptive Kalman
  filter. {\it J. Geod.},  80, 177-183.   

\bibitem{YL21} Yu, X. and Li, J. (2021) Adaptive Kalman filter for linear systems with
additive and multiplicative noises. {\it J. Latex Class Files.} 14, 8, 1-11.
\end{thebibliography}
   \end{document}